\documentclass[12pt]{amsart}
\usepackage{fullpage,colonequals}
\usepackage{amsmath,amsthm,amscd,amssymb,amsfonts}
\usepackage{mathtools,adforn,mathabx}
\usepackage{graphicx} 
\usepackage{xcolor} 
\usepackage[hidelinks]{hyperref}
\hypersetup{
  colorlinks,
  linkcolor={red!85!black},
  citecolor={blue!85!black},
  urlcolor={blue!95!black}
}
\usepackage[capitalize,nameinlink]{cleveref}

\renewcommand{\eqref}[1]{\hyperref[#1]{(\ref*{#1})}}
\crefformat{section}{#2\S#1#3}
\Crefformat{section}{#2\S#1#3}

\usepackage{mathrsfs,upgreek}
\usepackage[sans]{dsfont} 
\usepackage[math,condensed]{anttor}

\numberwithin{equation}{section} 
\numberwithin{figure}{section} 
\usepackage{enumitem} 
\setlist[enumerate]{label=$\arabic*.$, ref=$\arabic*$}

\theoremstyle{plain}
\newtheorem{theoalph}{Theorem}

\newtheorem{prop}{Proposition}[section]
\newtheorem{coro}[prop]{Corollary}
\newtheorem{lemm}[prop]{Lemma}

\newtheorem{conj}[prop]{Conjecture}

\newtheorem{innercustomthm}{Theorem}
\newenvironment{custtheo}[1]
{\renewcommand\theinnercustomthm{#1}\innercustomthm}
{\endinnercustomthm}

\theoremstyle{definition}

\theoremstyle{remark}

\newtheoremstyle{citing}
{3pt}
{3pt}
{\itshape}
{}
{\bfseries}
{.}
{.5em}
{\thmnote{#3}}

\theoremstyle{citing}

\newcommand{\R}{\mathbb{R}}

\newcommand{\Z}{\mathbb{Z}}

\newcommand{\cO}{\mathcal{O}}

\newcommand{\hX}{\widehat{X}}

\newcommand{\hlambda}{\widehat{\lambda}}

\newcommand{\hvarrho}{\widehat{\varrho}}

\newcommand{\hvarphi}{\widehat{\varphi}}

\newcommand{\tK}{\widetilde{K}}

\newcommand{\talpha}{\widetilde{\alpha}}

\renewcommand{\:}{\colon}
\renewcommand{\=}{\colonequals}
\newcommand{\dd}{\hspace{1pt}\operatorname{d}\hspace{-1pt}}

\DeclareMathOperator{\diam}{diam}

\renewcommand{\emph}[1]{\textsf{\textit{#1}}}

\newcommand{\aA}{\mathsf{A}}

\newcommand{\aX}{\mathsf{X}}

\newcommand{\whr}{\widehat{r}}
\newcommand{\whx}{\widehat{x}}
\newcommand{\whz}{\widehat{z}}
\newcommand{\wtm}{\widetilde{m}}

\newcommand{\hell}{\widehat{\ell}}

\let\il\int
\renewcommand{\int}{\text{ *** \emph{CHANGE command} *** }}

\newcommand{\OK}{{\mathcal{O}_K}}
\newcommand{\MK}{{\mathfrak{m}_K}}
\newcommand{\HK}{{\mathbb{H}_K}}
\newcommand{\PK}{{\mathbb{P}^{1}_K}}
\newcommand{\PKber}{{\mathsf{P}^{1}_K}}

\newcommand{\xcan}{{x_{\operatorname{can}}}}

\DeclareMathOperator{\crit}{crit}

\DeclareMathOperator{\ord}{ord}
\newcommand{\wdeg}[1]{\deg_{\operatorname{i}, #1}} 
\newcommand{\wmax}{\deg_{\operatorname{i}, \max}} 
\newcommand{\wf}{\operatorname{i}} 

\usepackage{microtype}
\begin{document}

\title{Locating critical points attracted to $p$-adic attracting cycles}
\author{Juan Rivera-Letelier}
\address{Department of Mathematics, University of Rochester. Hylan Building, Rochester, NY~14627, U.S.A.}
\email{riveraletelier@gmail.com}
\urladdr{\url{http://rivera-letelier.org/}}
\date{\today}

\begin{abstract}
  In complex dynamics, a fundamental result of \textsc{Fatou} and \textsc{Julia} asserts that every attracting cycle of a rational map attracts a critical point.
  The analogous statement fails in non-\textsc{Archimedean} dynamics.
  For a non-\textsc{Archimedean} rational map, this paper establishes a sharp condition on the multiplier of an attracting cycle ensuring it attracts a critical point.
\end{abstract}

\maketitle


\section{Introduction}
\label{s:introdcution}

Throughout this paper, ${K}$ is an algebraically closed field that is complete with respect to a nontrivial ultrametric norm~$|\cdot|$, and~$p$ is the residue characteristic of~$K$.
A cycle~$\cO$ of a rational map~$R$ with coefficients in~$K$ \emph{attracts a point~$x_0$ of~$\PK$}, if the forward orbit ${(R^{n \# \cO}(x_0))_{n = 1}^{+\infty}}$ of~$x_0$ under~$R^{\#\cO}$ converges to a point in~$\cO$.

In the complex setting, a fundamental result of \textsc{Fatou} \cite[\S30]{Fat20a} and \textsc{Julia} \cite[\S27]{Jul18} asserts that every attracting cycle attracts a critical point.
\textsc{Singer} \cite[Theorem~2.7]{Sin78} proved a similar result for interval maps with negative \textsc{Schwarzian} derivative.

The analogous statement fails in the non-\textsc{Archimedean} setting.
When ${p > 0}$ and~$K$ is of characteristic zero, the polynomial~$z^p$ provides the canonical examples.
If~$\cO$ is a cycle of~$z^p$ different from~$\{ 0 \}$ and~$\{ \infty \}$, then its multiplier is~$p^{\# \cO}$ and it is thus attracting.
But~$\cO$ attracts no critical point because the only critical points are~$0$ and~$\infty$ and they are fixed.

The following theorem shows that these examples are extreme: Every attracting cycle with a smaller multiplier must attract a critical point.
For each integer~$d$ satisfying ${d \ge 2}$, put
\begin{equation}
  \label{eq:1}
  \lambda(d)
  \=
  \min \{ |m| \: m \in \{ 1, \ldots, d \} \},
\end{equation}
where the norm of~$m$ is computed in~$K$.
Note that ${0 \le \lambda(d) \le 1}$, that ${\lambda(d) = 0}$ holds if and only if the characteristic of~$K$ is positive and less than or equal to~$d$, and that ${\lambda(d) = 1}$ holds if and only if the characteristic of~$K$ is zero or strictly larger than~$d$.

\begin{theoalph}
  \label{t:critically-attracted}
  Let~$R$ be a rational map of degree~$d$ at least two with coefficients in~$K$, and~$\cO$ an attracting cycle of~$R$.
  If the multiplier~$\lambda$ of~$\cO$ satisfies ${|\lambda| < \lambda(d)^{\#\cO}}$, then~$\cO$ attracts a critical point of~$R$.
\end{theoalph}

\cref{t:critically-attracted} is sharp if ${\lambda(d) > 0}$, see~\cref{ss:A-sharpness}.
When ${\lambda(d) = 0}$, the hypothesis of \cref{t:critically-attracted} is never satisfied.
In this case, there exist attracting cycles of every period and arbitrarily small multiplier attracting no critical point, see~\cref{ss:A-sharpness}.
When ${\lambda(d) = 1}$, \cref{t:critically-attracted} follows from \cite[Theorem~1.5]{BenIngJonLev14}.

In terms of the \emph{\textsc{Lyapunov} exponent~$\chi$ of~$\cO$}, defined by
\begin{equation}
  \label{eq:2}
  \chi
  \=
  \frac{1}{\# \cO} \log|\lambda|,
\end{equation}
the hypothesis of \cref{t:critically-attracted} corresponds to ${\chi < \log \lambda(d)}$.
Thus, the following corollary is a direct consequence of \cref{t:critically-attracted}.

\begin{coro}
  \label{c:critical-bound}
  Let~$R$ be a rational map of degree~$d$ at least two with coefficients in~$K$.
  Then, the number of cycles of~$R$ whose \textsc{Lyapunov} exponent is strictly less than~$\log \lambda(d)$ is finite, less than or equal to the number of critical points of~$R$.
\end{coro}

More generally, let~$\nu$ be a \textsc{Borel} probability measure on the \textsc{Berkovich} projective line~$\PKber$ that is invariant by~$R$, and denote by~$\| R' \|$ the spherical derivative of~$R$.
The \emph{\textsc{Lyapunov} exponent~$\chi_{\nu}(R)$ of~$\nu$}, is defined by
\begin{equation}
  \label{eq:3}
  \chi_{\nu}(R)
  \=
  \il \log \| R' \| \dd \nu.
\end{equation}
Since~$\| R' \|$ is bounded from above, the integral is defined and~$\chi_{\nu}(R)$ is either a real number or~$-\infty$.
If~$\cO$ is a cycle of~$R$ contained in~$\PK$ and~$\nu_{\cO}$ is the probability measure equidistributed on~$\cO$, then~$\exp(\# \cO \chi_{\nu_\cO}(R))$ is the norm of the multiplier of~$\cO$.

\textsc{Favre} and the author showed that the \textsc{Lyapunov} exponent of every ergodic measure is greater than or equal to~$\log \lambda(d)$, unless it is supported on an attracting cycle.
See \cite[\emph{Corollaire}~3.7]{0FavRiv25}, and \cite{Jac19,Nie2202} for related results.
Together with \cref{t:critically-attracted}, this yields the following corollary as a direct consequence.

\begin{coro}
  \label{c:critically-attracted}
  Let~$R$ be a rational map of degree~$d$ at least two with coefficients in~$K$.
  Then, every \textsc{Borel} probability measure~$\nu$ on~$\PKber$ that is ergodic and invariant by~$R$ satisfies
  \begin{equation}
    \label{eq:4}
    \chi_{\nu}(R)
    \ge
    \log \lambda(d),
  \end{equation}
  unless~$\nu$ is supported on a cycle that attracts a critical point.
  In particular, the number of those~$\nu$ for which~\eqref{eq:4} fails is finite, less than or equal to the number of critical points of~$R$.
\end{coro}

The next result improves the criterion of \textsc{Benedetto}, \textsc{Ingram}, \textsc{Jones}, and \textsc{Levy} \cite[Theorem~1.2]{BenIngJonLev14} for a rational map to be post-critically infinite.

\begin{theoalph}
  \label{t:infinitely-attracted}
  Let~$R$ be a rational map of degree~$d$ at least two with coefficients in~$K$, and~$\cO$ an attracting cycle of~$R$ of \textsc{Lyapunov} exponent~$\chi$.
  Then,
  \begin{equation}
    \label{eq:5}
    -\infty
    <
    \chi
    <
    \begin{cases}
      d \log \lambda(d)
      & \text{if } \# \cO \le 2d - 2;
      \\
      \left(1 + \frac{2(d - 1)^2}{\# \cO} \right) \log \lambda(d)
      & \text{if } \# \cO > 2d - 2
    \end{cases}
  \end{equation}
  implies that~$\cO$ attracts a critical point of~$R$ whose forward orbit is infinite.
\end{theoalph}

When the period~$\# \cO$ of~$\cO$ is large, the upper bound in~\eqref{eq:5} is close to~$\log \lambda(d)$, which is the upper bound in the hypothesis of \cref{t:critically-attracted}.
Hence, \cref{t:infinitely-attracted} is nearly sharp if~$\#\cO$ is large.
When~$\cO$ consists of a fixed point, \cref{t:infinitely-attracted} coincides with \cite[Theorem~1.2]{BenIngJonLev14}.
Applying the latter to the iterates of~$R$ yields a result analogous to \cref{t:infinitely-attracted}, but under a more restrictive hypothesis.
Indeed, for~$n$ in~$\Z_{> 0}$, applying \cite[Theorem~1.2]{BenIngJonLev14} to~$R^n$ requires
\begin{equation}
  \label{eq:156}
  \chi
  <
  \frac{d^n}{n} \log \lambda(d^n),
\end{equation}
which is more restrictive than ${\chi < d \log \lambda(d)}$, and hence than the upper bound in~\eqref{eq:5}.

\cref{t:infinitely-attracted'} in~\cref{s:critically-attracted} improves the upper bound in the hypothesis~\eqref{eq:5} of \cref{t:infinitely-attracted}, expressing it in terms of the number of critical points attracted to~$\cO$ and the degree of~$R$ on the immediate basin of~$\cO$.
\cref{t:infinitely-attracted'} also yields the following improvement of \cref{t:infinitely-attracted} for polynomials.
For each integer~$d$ satisfying ${d \ge 2}$, put
\begin{equation}
  \label{eq:6}
  \hlambda(d)
  \=
  \min \{ |m|^m \: m \in \{1, \ldots, d \}\}.
\end{equation}

\begin{coro}
  \label{c:infinitely-attracted'}
  Let~$P$ be a polynomial of degree~$d$ at least two with coefficients in~$K$, and~$\cO$ an attracting cycle of~$P$ of \textsc{Lyapunov} exponent~$\chi$.
  Then,
  \begin{equation}
    \label{eq:7}
    -\infty
    <
    \chi
    <
    \begin{cases}
      \log \hlambda(d)
      & \text{if } \# \cO \le d - 1;
      \\
      \log \lambda(d) + \frac{d - 1}{\# \cO} \log \frac{\hlambda(d)}{\lambda(d)}
      & \text{if } \# \cO > d - 1
    \end{cases}
  \end{equation}
  implies that~$\cO$ attracts a critical point of~$P$ whose forward orbit is infinite.
\end{coro}

Since ${\hlambda(d) \ge \lambda(d)^d}$, the hypothesis on~$\chi$ in \cref{c:infinitely-attracted'} is weaker than that in \cref{t:infinitely-attracted}.
When~$\cO$ consists of a fixed point, \cref{c:infinitely-attracted'} coincides with \cite[Theorems~5.1]{BenIngJonLev14}.

The proofs of Theorems~\ref{t:critically-attracted} and~\ref{t:infinitely-attracted} rely on (a more precise version of) the following result in non-\textsc{Archimedean} analysis.

\begin{theoalph}
  \label{t:critically-mapped}
  Let~$d$ be an integer satisfying ${d \ge 2}$ and ${\lambda(d) > 0}$, and~$Q$ a rational map of degree~$d$ with coefficients in~$K$ that is not a polynomial and satisfies ${Q(\infty) = \infty}$.
  Let~$z_0$ be in~$K$ that is not a pole of~$Q$, and denote by~$r$ the shortest norm distance from~$z_0$ to a pole of~$Q$.
  Then, $Q$ has a critical value~$v$ satisfying
  \begin{equation}
    \label{eq:8}
    |v - Q(z_0)|
    \le
    \lambda(d)^{-1} |Q'(z_0)| r.
  \end{equation}
\end{theoalph}

\cref{t:critically-mapped} is sharp if~$d$ is not a power of~$p$, see~\cref{ss:C-sharpness}.

The following corollary provides an upper bound for ${|v - Q(z_0)|}$ in terms of the shortest norm distance from~$z_0$ to a pole or preimage of~$Q(z_0)$ by~$Q$ distinct from~$z_0$.
It follows from \cref{t:critically-mapped} together with power series computations.
For an integer~$d$ satisfying ${d \ge 2}$ and ${\lambda(d) > 0}$, observe that ${\lambda(d) < 1}$ implies ${p > 0}$ and ${0 < |p| < 1}$, and define
\begin{equation}
  \label{eq:9}
  \gamma(d)
  \=
  \begin{cases}
    1
    & \text{if } \lambda(d) = 1;
    \\
    |p|^{-\frac{1}{p - 1}}
    & \text{if } \lambda(d) < 1.
  \end{cases}
\end{equation}

\begin{coro}
  \label{c:critically-mapped}
  Let~$d$ be an integer satisfying ${d \ge 2}$ and ${\lambda(d) > 0}$, and~$Q$ a rational map of degree~$d$ with coefficients in~$K$ satisfying ${Q(\infty) = \infty}$.
  Let~$z_0$ be in~$K$ that is not a pole of~$Q$, and denote by~$r_{\bullet}$ the shortest norm distance from~$z_0$ to a pole or a preimage of~$Q(z_0)$ by~$Q$ distinct from~$z_0$.
  Then, ${Q}$ has a critical value~$v$ satisfying
  \begin{equation}
    \label{eq:10}
    |v - Q(z_0)|
    \le
    \gamma(d) \lambda(d)^{-1} |Q'(z_0)| r_{\bullet}.
  \end{equation}
\end{coro}

\cref{c:critically-mapped} is sharp, see~\cref{ss:C-sharpness}.
The upper bound in~\eqref{eq:10} sharpens the bound~$|\lambda(d)|^{-d} |Q'(z_0)| r_{\bullet}$ from \cite[Theorem~1.4]{BenIngJonLev14}.
The proof of \cite[Theorem~1.4]{BenIngJonLev14} yields a variant of \cref{c:critically-mapped} ensuring ${v \neq Q(z_0)}$, but with a larger upper bound for ${|v - Q(z_0)|}$.
See \cref{t:strictly-critically-mapped''} in~\cref{s:critically-mapped}, which is used in the proof of \cref{t:infinitely-attracted}.

The following variant of \cref{t:critically-mapped} also ensures ${v \neq Q(z_0)}$ and (a more precise version of it) is used in the proof of \cref{t:infinitely-attracted}.

\begin{theoalph}
  \label{t:strictly-critically-mapped}
  Let~$d$, $Q$, $z_0$, and~$r$ be as in \cref{t:critically-mapped}.
  Then, $Q$ has a critical value~$v$ satisfying
  \begin{equation}
    \label{eq:11}
    0
    <
    |v - Q(z_0)|
    \le
    \lambda(d)^{-(d - 1)} |Q'(z_0)| r.
  \end{equation}
\end{theoalph}

\subsection{Notes and references}
\label{ss:notes-references}
In the complex setting, the result of \textsc{Fatou} and \textsc{Julia} on critical points in immediate basins implies that a rational map of degree~$d$ at least two has at most~${2d - 2}$ attracting cycles.\footnote{\textsc{Shishikura} \cite[Corollary~1]{Shi87} proved that the number non-repelling cycles is also bounded by~${2d - 2}$. Even when restricted to indifferent cycles, the analogous statement fails in the non-\textsc{Archimedean} setting: The number of indifferent cycles is either zero or infinity if ${p > 0}$ \cite[\emph{Corollaire}~5.17]{Riv03c}.}
In the non-\textsc{Archimedean} setting, the analogous statement fails because rational maps often have infinitely many attracting cycles if ${p > 0}$.
However, a weaker statement holds: If the number of attracting cycles is finite, then it is at most~${3d - 3}$ \cite[\emph{Corollaire}~4.9]{Riv03c}.
A positive solution to the following conjecture would yield the optimal upper bound of ${2d - 2}$.
See~\cref{s:critically-attracted} for the definition of immediate basin (of \textsc{Cantor} type).
A rational map can have at most ${d - 1}$ attracting cycles whose immediate basins are of \textsc{Cantor} type \cite[\emph{Proposition}~4.8]{Riv03c}.
No immediate basin of a polynomial is of \textsc{Cantor} type, except, perhaps, for that of~$\infty$.

\begin{conj}[\textcolor{black}{\cite[p.~197]{Riv03c}}]
  Let~$R$ be a rational map of degree at least two with coefficients in~$K$.
  Suppose that~$R$ has an attracting cycle whose immediate basin is of \textsc{Cantor} type and contains no critical point of~$R$.
  Then, $R$ has infinitely many attracting cycles.
\end{conj}

To solve this conjecture affirmatively, it would be sufficient to prove that every immediate basin of \textsc{Cantor} type containing no critical point has an inseparable periodic point in its boundary \cite[\emph{Lemme}~4.1]{Riv05b}.

For each rational map~$R$ of degree~$d$ at least two with coefficients in~$K$, \textsc{Gauthier}, \textsc{Okuyama}, and \textsc{Vigny} \cite[Theorem~A]{GauOkuVig20} approximated the \textsc{Lyapunov} exponent of~$R$ in terms of the \textsc{Lyapunov} exponents of periodic measures of a given period~$n$ that are not too negative.
For every~$r$ in the interval~$]0, \lambda(d)^{d^n}]$, their estimate holds with the cutoff~$\frac{1}{n} \log r$.
\cref{c:critical-bound} implies that \cite[Theorem~A]{GauOkuVig20} holds for every~$r$ in the larger interval~$]0, \lambda(d)^n]$.
\cref{c:critical-bound'}, a refined version \cref{c:critical-bound}, leads to a further sharpening.

\textsc{Faber} \cite[Applications~1 and~2]{Fab13b} established variants of Theorems~\ref{t:critically-mapped} and~\ref{t:strictly-critically-mapped} locating a critical point instead of a critical value.

\subsection{Strategy and organization}
\label{ss:organization}
After some preliminaries in~\cref{s:preliminaries}, \cref{s:critically-mapped} states and proves Theorems~\ref{t:critically-mapped'} and~\ref{t:strictly-critically-mapped'}, which are refined versions of Theorems~\ref{t:critically-mapped} and~\ref{t:strictly-critically-mapped}, and a variant of these results implicit in \cite{BenIngJonLev14} (\cref{t:strictly-critically-mapped''}).
The proofs of Theorems~\ref{t:critically-mapped'} and~\ref{t:strictly-critically-mapped'} are both divided into two cases, according to whether the poles of~$Q$ are large or small.
In the former case, the results reduce to statements locating critical points in the absence of poles in~\cref{ss:critically-holomorphic}.
The proofs in the latter case parallel that of \cite[Theorem~4.1]{BenIngJonLev14}.

\cref{s:critically-attracted} states and proves Theorems~\ref{t:critically-attracted'} and~\ref{t:infinitely-attracted'}, which are refined version of Theorems~\ref{t:critically-attracted} and~\ref{t:infinitely-attracted}.
Roughly speaking, the proof of \cref{t:critically-attracted'} identifies a point~$z_0$ of the attracting cycle~$\cO$ at which~$R$ is most contracting relative to the inner radii of the immediate basin of~$\cO$ at~$z_0$ and~$R(z_0)$.
After a suitable coordinate change, \cref{t:critically-mapped'} locates the desired critical point.
The proof of~\cref{t:infinitely-attracted'} relies on \cref{t:critically-mapped'} to choose a point~$z_0$ of~$\cO$, for which the following holds for every~$n$ in~$\Z_{> 0}$: There is~$j_n$ in~$\{0, \ldots, n - 1\}$ and a critical value~$v_n$ of~$R$ in~$\PK$ that is close to, but distinct from, $R^{j_n + 1}(z_0)$ and such that~$R^{n - j_n - 1}$ is univalent on the smallest ball contained in the immediate basin of~$\cO$ containing~$R^{j_n + 1}(z_0)$ and~$v_n$.
The proof in the base case ${n = 1}$ follows from \cref{t:strictly-critically-mapped'} and \cref{l:critically-holomorphic'}, while that of the inductive step relies on \cref{t:strictly-critically-mapped''} in~\cref{s:critically-mapped} and a general lemma (\cref{l:injectivity-radius}).
\cref{t:infinitely-attracted'} then follows from the finiteness of the set of critical points of~$R$ in~$\PK$.
The proofs of Theorems~\ref{t:critically-attracted'} and~\ref{t:infinitely-attracted'} occupy~\S\S\ref{ss:proof-critically-attracted'} and~\ref{ss:proof-infinitely-attracted'}, respectively.

\cref{s:examples} provides examples showing that Theorems~\ref{t:critically-attracted} and~\ref{t:critically-mapped}, and \cref{c:critically-mapped} are sharp.

\section{Preliminaries}
\label{s:preliminaries}
Denote by ${\diam \: \HK \to \R_{\ge 0} \cup \{+ \infty\}}$ the affine diameter function.
If ${I \: \PK \to \PK}$ is the involution defined in affine coordinates by ${I(z) \= 1/z}$, then~$\diam$ coincides with~$\diam$ on ${\{x \in \PKber \: |x| \le 1\}}$ and with~$1/\diam \circ I$ on ${\{x \in \PKber \: |x| > 1\}}$.

Let~$R$ be a nonconstant rational map with coefficients in~$K$.
For each connected subset~$A'$ of~$\PKber$ and every connected component~$A$ of~$R^{-1}(A')$, the function
\begin{equation}
  \label{eq:12}
  y \mapsto \sum_{x \in A \cap R^{-1}(y)} \deg_R(x)
\end{equation}
is constant \cite[\emph{Lemme}~1.1]{0FavRiv25}.
Denote by~$\deg_R(A)$ the value of this function.

The \emph{maximal inseparability degree~$\wmax(R)$ of~$R$}, is
\begin{equation}
  \label{eq:13}
  \wmax(R)
  \=
  \max \{ \wdeg{R}(x) \: x \in \HK \}.
\end{equation}
Note that~$\wmax(R)$ is a power of~$p$ satisfying ${\wmax(R) \le \deg(R)}$.

Let ${\wf_R \: \HK \to \R}$ be defined by
\begin{equation}
  \label{eq:14}
  \wf_R
  \=
  -\log (\| R' \times \| \diam / \diam \circ R).
\end{equation}
For every nonconstant rational map~$Q$ with coefficients in~$K$, the equality ${\wf_{Q \circ R} = \wf_Q \circ R + \wf_R}$ holds.
On the other hand, for every \textsc{M{\"o}bius} transformation~$\varphi$ the function~$\wf_\varphi$ is constant equal to~$0$ and ${\wf_{\varphi \circ R \circ \varphi^{-1}} = \wf_R \circ \varphi^{-1}}$.
By~\cite[\emph{Proposition}~3.4]{0FavRiv25}, ${|\wmax(R)| > 0}$ implies
\begin{equation}
  \label{eq:15}
  0
  \le
  \wf_R
  \le
  -\log|\wdeg{R}|.
\end{equation}

\subsection{Distorted log-size}
\label{ss:distorded-log-size}
This section reviews basic properties of a one-parameter family of functions that play an important r{\^o}le in the proofs of \cref{l:critically-holomorphic'} and Theorems~\ref{t:critically-mapped'} and~\ref{t:strictly-critically-mapped'} in~\cref{s:critically-mapped}.
These functions are variants of a function introduced in the proof of \cite[Theorem~4.1]{BenIngJonLev14}.
For each~$x$ in~$\HK$ of type~II or~III, denote by~$B(x)$ the ball of~$K$ associated with it.

Let~$R$ be a nonconstant rational map with coefficients in~$K$.
Given a subset~$W$ of~$\PK$, denote by~$Z(R, W)$ and~$P(R, W)$ the number of zeros and poles of~$R$ in~$W$, respectively, counted with multiplicity.
Given~$d$ in~$\R$, let ${G_d \: \HK \to \R}$ be the function defined by
\begin{equation}
  \label{eq:16}
  G_d(x)
  \=
  \log \diam(R(x)) + d \wf_R(x).
\end{equation}
For each~$x$ in~$\HK$ of type~II or~III, define~$\partial G_d(x)$ as follows.
Fix~$z$ in~$B(x)$ and for each~$t$ in~$\R$ denote by~$x_0(t)$ the point of~$\PKber$ associated with ${\{ z' \in K \: |z' - z| \le \exp(t) \}}$.
Then, $\partial G_d(x)$ is the right derivative of~$G_d \circ x_0$ at ${t = \log \diam(x)}$.
Since the restriction of~$x_0$ to~$[\log \diam(x), +\infty[$ is independent of the choice of~$z$, so is~$\partial G_d(x)$.

\begin{lemm}
  \label{l:distorded-log-size}
  For all~$d$ in~$\R$ and~$x$ in~$\HK$ of type~II or~III satisfying ${R([x, \infty[) \subseteq ]0, \infty[}$,
  \begin{equation}
    \label{eq:17}
    \partial G_d(x)
    =
    (d + 1)(Z(R, B(x)) - P(R, B(x))) - d(Z(R', B(x)) - P(R', B(x)) + 1).
  \end{equation}
\end{lemm}

\begin{proof}
  Let~$z$ and ${x_0 \: \R \to ]z, \infty[}$ be as above and put ${t_0 \= \log \diam(x)}$.
  The hypothesis ${R([x, \infty[) \subseteq ]0, \infty[}$ implies
  \begin{equation}
    \label{eq:18}
    \diam \circ R \circ x_0
    =
    |R \circ x_0|
  \end{equation}
  on ${[t_0, +\infty[}$.
  Combined with the definition of~$\wf_R$, this implies that, for every~$t$ in ${[t_0, +\infty[}$,
  \begin{equation}
    \label{eq:19}
    \partial G_d(x)(x_0(t))
    =
    (d + 1)\log|R(x_0(t))| - d\log|R'(x_0(t))| - dt.
  \end{equation}
  The desired assertion follows from an analysis of the \textsc{Newton} polygons of~$R$ and~$R'$.
\end{proof}

\subsection{Critical points in the absence of poles}
\label{ss:critically-holomorphic}

This section gathers results locating critical points of rational maps in regions without poles.
Several arguments in \S\S\ref{s:critically-mapped} and~\ref{s:critically-attracted} rely on the following lemma.

\begin{lemm}
  \label{l:critically-holomorphic}
  Let~$Q$ be a nonconstant rational map with coefficients in~$K$, and let~$r$ and~$r'$ in~$\R_{> 0}$ and~$d$ in~$\Z_{> 0}$ be such that~$Q$ maps~$\{ z \in K \: |z| < r \}$ onto~$\{ z \in K \: |z| < r' \}$ with degree~$d$.
  Then, ${|Q'(0)| < |d| r'/r}$ implies that~$Q$ has a critical point in~$\{ z \in K \: |z| < r \}$.
\end{lemm}

\begin{proof}
  Suppose~$Q$ has no critical point in~$\{ z \in K \: |z| < r \}$ and let~$\sum_{n = 0}^{+\infty} a_n z^n$ be the power series expansion of~$Q$ on this set.
  Then, for every~$n$ in~$\Z_{\ge 0}$, the inequality ${|n a_n| r^n \le r'}$ holds, with equality for ${n = d}$.
  Analyzing the \textsc{Newton} polygon of~$Q'$ yields that, for every~$n$ in~$\Z_{> 0}$, the inequality ${|n a_n| r^{n - 1} \le |a_1|}$ holds.
  Taking ${n = d}$, yields
  \begin{equation}
    \label{eq:20}
    |Q'(0)|
    =
    |a_1|
    \ge
    |d a_d| r^{d - 1}
    =
    |d| r'/r.
    \qedhere
  \end{equation}
\end{proof}

The proofs of \cref{t:strictly-critically-mapped'} in~\cref{s:critically-mapped} and \cref{t:infinitely-attracted'} in~\cref{s:critically-attracted}, rely on the following lemma.

\begin{lemm}
  \label{l:critically-holomorphic'}
  Let~$Q$ be a rational map with coefficients in~$K$ such that ${Q'(0) \neq 0}$, let~$r$ and~$r'$ in~$\R_{> 0}$ be such that~$Q$ maps ${\{z \in K \: |z| < r \}}$ onto~$\{ z \in K \: |z| < r' \}$, and put ${d \= \deg_Q(\{z \in K \: |z| < r \})}$.
  Suppose either
  \begin{align}
    \label{eq:21}
    |Q'(0)|
    & <
      \exp(-(d - 1)\wf_Q(x(r))) r'/r,
      \intertext{or, denoting by~$r_{\bullet}$ the largest number in~$]0, r]$ such that~$Q$ is univalent on ${\{z \in K \: |z| < r_{\bullet}\}}$,}
      \label{eq:22}
      |Q'(0)|
    & <
      \exp(-d\wf_Q(x(r))) r'/r_{\bullet}.
  \end{align}
  Then, $Q$ has a critical point in~$\{ z \in K \: |z| < r \}$ that is not a zero.
\end{lemm}

Under the hypothesis~\eqref{eq:22}, this lemma follows from the proof of \cite[Theorem~5.1]{BenIngJonLev14}.
For completeness, a proof follows.

\begin{proof}[Proof of \cref{l:critically-holomorphic'}]
  For each~$\varrho$ in~$\R_{> 0}$, denote by~$x(\varrho)$ the point of~$\PKber$ associated with ${\{z \in K \: |z| \le \varrho\}}$.
  As in~\cref{ss:distorded-log-size}, for each~$x$ in~$\HK$ of type~II or~III, denote by~$B(x)$ the ball of~$\PK$ associated with it.

  Note that
  \begin{equation}
    \label{eq:23}
    \wdeg{Q}(x(r_{\bullet}))
    =
    1,
    \wf_Q(x(r_{\bullet}))
    =
    0,
    \text{ and }
    |Q'(0)| r_{\bullet}
    =
    \diam(Q(x(r_{\bullet})))
    \ge
    \left( \frac{r_{\bullet}}{r} \right)^d r'.
  \end{equation}
  Hence,
  \begin{equation}
    \label{eq:24}
    \left( \frac{|Q'(0)| r_{\bullet}}{r'} \right)^{d - 1}
    \le
    \left( \frac{|Q'(0)| r}{r'} \right)^d,
  \end{equation}
  and thus~\eqref{eq:21} implies~\eqref{eq:22}.

  Suppose~\eqref{eq:22} holds and let~$G_{-d}$ be the function defined in~\cref{ss:distorded-log-size} with ${R = Q}$ and~$d$ replaced by~$-d$.
  By~\eqref{eq:23},
  \begin{equation}
    \label{eq:25}
    G_{-d}(x(r_{\bullet}))
    =
    \log \diam(Q(x(r_{\bullet})))
    =
    |Q'(0)| r_{\bullet}
    <
    -d \wf_Q(x(r)) + \log r'
    =
    G_{-d}(x(r)).
  \end{equation}
  It follows that ${r_{\bullet} < r}$ and that there is~$r_{\ddag}$ in ${]r_{\bullet}, r[}$ such that ${\partial G_{-d}(x(r_{\ddag})) > 0}$.
  Put
  \begin{equation}
    \label{eq:26}
    B_{\ddag} \= \{z \in K \: |z| \le r_{\ddag}\}
    \text{ and }
    d_{\ddag}
    \=
    \deg_Q(B_{\ddag}).
  \end{equation}
  From the definition of~$r_{\bullet}$,
  \begin{equation}
    \label{eq:27}
    d
    \ge
    d_{\ddag}
    \ge
    2.
  \end{equation}
  On the other hand, by \cref{l:distorded-log-size} with~$d$ replaced by~$-d$,
  \begin{equation}
    \label{eq:28}
    0
    <
    G_{-d}(x(r_{\ddag}))
    =
    -(d - 1)(Z(Q, B_{\ddag}) - P(Q, B_{\ddag})) + d(Z(Q', B_{\ddag}) - P(Q', B_{\ddag}) + 1).
  \end{equation}
  In particular, $\partial G_{-d}(x(r_{\ddag}))$ is an integer and, thus, it is greater than or equal to~$1$.
  Since ${Z(Q, B_{\ddag}) = d_{\ddag}}$ and neither~$Q$ nor~$Q'$ has a pole in~$B_{\ddag}$,
  \begin{equation}
    \label{eq:29}
    d Z(Q', B_{\ddag})
    \ge
    (d - 1) (d_{\ddag} - 1).
  \end{equation}
  Denote by~$Z_{\S}$ the number of critical points of~$Q$ in~$B_{\ddag}$ that are not zeros of~$Q$, counted with multiplicity.
  The hypothesis ${|\wmax(Q)| > 0}$ implies that each zero~$z_0$ of~$Q$ satisfies
  \begin{equation}
    \label{eq:30}
    p
    \not\mid
    \deg_Q(z_0)
    \text{ and }
    \deg_Q(z_0)
    =
    \ord_{Q'}(z_0) + 1.
  \end{equation}
  Thus,
  \begin{equation}
    \label{eq:31}
    d_{\ddag}
    =
    \sum_{z_0 \in Q^{-1}(0) \cap B_{\ddag}} (\ord_{Q'}(z_0) + 1)
    =
    Z(Q', B_{\ddag}) - Z_{\S} + \# (Q^{-1}(0) \cap B_{\ddag}).
  \end{equation}
  Together with~\eqref{eq:27} and~\eqref{eq:29}, this implies
  \begin{equation}
    \label{eq:32}
    d Z_{\S}
    \ge
    d \# (Q^{-1}(0) \cap B_{\ddag}) - d_{\ddag} - (d - 1)
    \ge
    1 + d(\# (Q^{-1}(0) \cap B_{\ddag}) - 2).
  \end{equation}
  Since ${d_{\ddag} \ge 2}$ and~$0$ is a simple zero of~$Q$ because ${Q'(0) \neq 0}$,
  \begin{equation}
    \label{eq:33}
    \# (Q^{-1}(0) \cap B_{\ddag})
    \ge
    2,
    d Z_{\S}
    \ge
    1,
    \text{ and }
    Z_{\S}
    \ge
    1.
  \end{equation}
  Since ${B_{\ddag} \subseteq \{z \in K \: |z| < r\}}$, it follows that~$Q$ has a critical point in ${\{z \in K \: |z| < r\}}$ that is not a zero.
\end{proof}

\section{Locating critical values of rational functions}
\label{s:critically-mapped}

This section states and proves refined versions of Theorems~\ref{t:critically-mapped} and~\ref{t:strictly-critically-mapped} and \cref{c:critically-mapped} from~\cref{s:introdcution}, as well as a variant of these results implicit in \cite{BenIngJonLev14} (\cref{t:strictly-critically-mapped''}).
Recall from~\cref{s:preliminaries} that the maximal inseparability degree of a nonconstant rational map~$Q$ with coefficients in~$K$, is defined by
\begin{equation}
  \label{eq:34}
  \wmax(Q)
  =
  \max \{ \wdeg{Q}(x) \: x \in \HK \}.
\end{equation}

The proofs of Theorems~\ref{t:critically-attracted'} and~\ref{t:infinitely-attracted'}, stated in~\cref{s:critically-attracted}, rely on the following theorem.
Since ${\lambda(\deg(Q)) \le |\wmax(Q)|}$, \cref{t:critically-mapped} in~\cref{s:introdcution} is a direct consequence of the following theorem with~$Q(z)$ replaced by ${Q(z + z_0) - Q(z_0)}$.

\begin{custtheo}{C'}
  \label{t:critically-mapped'}
  Let~$Q$ be a rational map of degree at least two with coefficients in~$K$ that is not a polynomial and satisfies ${Q(\infty) = \infty}$, ${Q(0) = 0}$, ${Q'(0) \neq 0}$, and ${|\wmax(Q)| > 0}$.
  Furthermore, denote by~$r$ the smallest norm of a pole of~$Q$.
  Then, $Q$ has a critical value~$v$ in~$K$ satisfying
  \begin{equation}
    \label{eq:35}
    |v|
    \le
    |\wmax(Q)|^{-1} |Q'(0)| r.
  \end{equation}
  More precisely, ${Q}$ has a critical point in the component of
  \begin{equation}
    \label{eq:36}
    Q^{-1}(\{ z \in K \: |z| \le |\wmax(Q)|^{-1} |Q'(0)| r \})
  \end{equation}
  containing~$0$.
\end{custtheo}

The following corollary of \cref{t:critically-mapped'} provides an upper bound for~$|v|$ in terms of the smallest norm of a pole or zero of~$Q$ distinct from~$0$.
It follows from \cref{t:critically-mapped'} together with power series computations.
Since ${\lambda(\deg(Q)) \le |\wmax(Q)|}$, \cref{c:critically-mapped} is a direct consequence of the following corollary with~$Q(z)$ replaced by ${Q(z + z_0) - Q(z_0)}$.

\begin{coro}
  \label{c:critically-mapped'}
  Let~$Q$ be a rational map of degree at least two with coefficients in~$K$ satisfying ${Q(\infty) = \infty}$, ${Q(0) = 0}$, ${Q'(0) \neq 0}$, and ${|\wmax(Q)| > 0}$, and put
  \begin{equation}
    \label{eq:37}
    \gamma(Q)
    \=
    \begin{cases}
      1
      & \text{if } \wmax(Q) = 1;
      \\
      |p|^{-\frac{1}{p - 1}}
      & \text{if } \wmax(Q) > 1.
    \end{cases}
  \end{equation}
  Furthermore, denote by~$r_{\bullet}$ the smallest norm of a pole or a zero of~$Q$ distinct from~$0$.
  Then, ${Q}$ has a critical value~$v$ in~$K$ satisfying
  \begin{equation}
    \label{eq:38}
    |v|
    \le
    \gamma(Q) |\wmax(Q)|^{-1} |Q'(0)| r_{\bullet}.
  \end{equation}
\end{coro}

When ${\wmax(Q) > 1}$, the hypothesis ${|\wmax(Q)| > 0}$ implies ${p > 0}$ and ${0 < |p| < 1}$, so~$\gamma(Q)$ is defined and satisfies ${\gamma(Q) > 1}$.

The following variant of \cref{t:critically-mapped'} locates a nonzero critical value, under a more restrictive hypothesis.
The proof of \cref{t:infinitely-attracted'}, stated in~\cref{s:critically-attracted}, relies on it.
Since ${\lambda(\deg(Q)) \le |\wmax(Q)|}$, \cref{t:strictly-critically-mapped} in~\cref{s:introdcution} is a direct consequence of the following theorem with~$Q(z)$ replaced by ${Q(z + z_0) - Q(z_0)}$ and ${\whr = \lambda(d)^{-(d - 1)} |Q'(z_0)| r}$.

\begin{custtheo}{D'}
  \label{t:strictly-critically-mapped'}
  Let~$Q$ and~$r$ be as in \cref{t:critically-mapped'}, let~$\whr$ be in~$\R_{> 0}$, and denote by~$X$ the component of ${Q^{-1}(\{ z \in K \: |z| \le \whr \})}$ containing~$0$.
  Then,
  \begin{equation}
    \label{eq:39}
    \whr
    \ge
    |\wmax(Q)|^{-(\deg_Q(X) - 1)} |Q'(0)| r
  \end{equation}
  implies that~$Q$ has a critical point in~$X$ that is not a zero.
\end{custtheo}

The following theorem follows from the proof of \cite[Theorems~1.4]{BenIngJonLev14}.
The proof of \cref{t:infinitely-attracted'}, stated in~\cref{s:critically-attracted}, relies on it.

\begin{theoalph}
  \label{t:strictly-critically-mapped''}
  Let~$Q$ be a rational map of degree at least two with coefficients in~$K$ satisfying ${Q(\infty) = \infty}$, ${Q(0) = 0}$, ${Q'(0) \neq 0}$, and ${|\wmax(Q)| > 0}$.
  Furthermore, denote by~$r_{\bullet}$ the smallest norm of a pole or a zero of~$Q$ distinct from~$0$, let~$\whr_{\bullet}$ be in~$\R_{> 0}$, and denote by~$X$ the component of ${Q^{-1}(\{ z \in K \: |z| \le \whr_{\bullet} \})}$ containing~$0$.
  Then,
  \begin{equation}
    \label{eq:40}
    \whr_{\bullet}
    \ge
    |\wmax(Q)|^{-\deg_Q(X)} |Q'(0)| r_{\bullet}
  \end{equation}
  implies that~$Q$ has a critical point in~$X$ that is not a zero.
\end{theoalph}

The following corollary is a direct consequence of the previous theorem.

\begin{coro}
  \label{c:strictly-critically-mapped''}
  Let~$Q$ be a rational map of degree at least two with coefficients in~$K$ satisfying ${Q(\infty) = \infty}$ and ${|\wmax(Q)| > 0}$.
  Let~$z_0$ be in~$K$ that is not a pole or a critical point of~$Q$, and denote by~$r_{\bullet}$ the shortest norm distance from~$z_0$ to a pole or a preimage of~$Q(z_0)$ by~$Q$ distinct from~$z_0$.
  Then, ${Q}$ has a critical value~$v$ in~$K$ satisfying
  \begin{equation}
    \label{eq:41}
    0
    <
    |v - Q(z_0)|
    \le
    |\wmax(Q)|^{-\deg(Q)} |Q'(z_0)| r_{\bullet}.
  \end{equation}
\end{coro}

The rest of this section proves Theorems~\ref{t:critically-mapped'}, \ref{t:strictly-critically-mapped'}, and~\ref{t:strictly-critically-mapped''}, and \cref{c:critically-mapped'}.
Each of the proofs of Theorems~\ref{t:critically-mapped'} and~\ref{t:strictly-critically-mapped'} splits into two cases, according to whether the poles of~$Q$ are large or small.
To describe these cases more precisely, denote by~$B_0$ the ball expected to contain a critical value of~$Q$ and by~$X_0$ the component of~$Q^{-1}(B_0)$  containing~$0$.
The first case occurs when~$X_0$ is a ball.
In this situation, \cref{t:critically-mapped'} follows from \cref{l:critically-holomorphic} and \cref{t:strictly-critically-mapped'} from \cref{l:critically-holomorphic'}.
The second case, when~$X_0$ is not a ball, parallels the proof of \cite[Theorem~4.1]{BenIngJonLev14}.
See also the proof of \cref{t:strictly-critically-mapped''} in~\cref{ss:proof-critically-mapped'''}.
A key step is to employ the distorted log-size, introduced in \cite{BenIngJonLev14} and reviewed in~\cref{ss:distorded-log-size}, to construct an affinoid to make residue computations.
The fact that~$X_0$ is not a ball is decisive.
It allows a distortion factor of ${\deg_R(X) - 1}$ in the proof of \cref{t:strictly-critically-mapped'}, leading to the exponent in the statement.
This improvement over the exponent~$\deg_R(X)$ in the statement of \cref{t:strictly-critically-mapped''} is crucial in the proof of \cref{t:infinitely-attracted'} in~\cref{s:critically-attracted}.
Likewise, the optimal distortion factor of~$1$ in the proof of \cref{t:critically-mapped'} is possible thanks to the fact that~$X_0$ is not a ball and the hypothesis, made without loss of generality, that every zero of~$Q$ in~$X_0$ is simple.

The proofs of \cref{t:critically-mapped'}, \cref{c:critically-mapped'}, and \cref{t:strictly-critically-mapped'} occupy \S\S\ref{ss:proof-critically-mapped}, \ref{ss:c-proof-critically-mapped'}, and \ref{ss:proof-critically-mapped'}, respectively.
As mentioned above, \cref{t:strictly-critically-mapped''} follows from the proof of \cite[Theorems~1.4]{BenIngJonLev14}.
\cref{ss:proof-critically-mapped'''} proves this result for completeness.

The rest of this section uses the following notation.
For each~$\varrho$ in~$\R_{> 0}$, denote by~$x(\varrho)$ the point of~$\PKber$ associated with ${\{z \in K \: |z| \le \varrho\}}$.
As in~\cref{ss:distorded-log-size}, for each~$x$ in~$\HK$ of type~II or~III, denote by~$B(x)$ the ball of~$K$ associated with it.

\subsection{Proof of \cref{t:critically-mapped'}}
\label{ss:proof-critically-mapped}
Put
\begin{equation}
  \label{eq:42}
  r_0
  \=
  |\wmax(Q)|^{-1} |Q'(0)| r,
  B_0
  \=
  \{z \in K \: |z| \le r_0\},
\end{equation}
and denote by~$X_0$ the component of~$Q^{-1}(B_0)$ containing~$0$.
The goal is to prove that~$Q$ has a critical point in~$X_0$.
To do this, put
\begin{equation}
  \label{eq:43}
  D
  \=
  \{z \in K \: |z| < r\},
  D'
  \=
  Q(D),
  \text{ and }
  r'
  \=
  \diam(D').
\end{equation}
Since~$D$ contains no pole of~$Q$, the set~$D'$ is a ball of~$K$ containing~$0$ and~$r'$ is finite.
If ${r' > r_0}$, then ${B_0 \subseteq D'}$, ${X_0 \subseteq D}$, and thus~$X_0$ is a ball of~$K$ satisfying ${\diam(X_0) < r}$.
On the other hand,
\begin{equation}
  \label{eq:44}
  |Q'(0)|
  =
  \frac{|\wmax(Q)| r_0}{r}
  <
  \frac{|\wmax(Q)| r_0}{\diam(X_0)},
\end{equation}
so \cref{l:critically-holomorphic} implies that~$Q$ has a critical point in~$X_0$.
For the remainder of the proof, suppose ${r' \le r_0}$, so ${D' \subseteq B_0}$ and ${D \subseteq X_0}$.

Note that~$r$ and~$r'$ are both in~$|K^{\times}|$.
Let~$\eta$ and~$\eta'$ in~$K$ be such that ${|\eta| = r}$ and ${|\eta'| = r'}$, and let~$\varphi$ and~$\hvarphi$ be the \textsc{M{\"o}bius} transformations given in affine coordinates by
\begin{equation}
  \label{eq:45}
  \varphi(z)
  \=
  \frac{\eta}{z}
  \text{ and }
  \hvarphi(z)
  \=
  \frac{\eta'}{z}.
\end{equation}
Moreover, denote by~$Q_{\dag}$ the rational map with coefficients in~$K$ defined by ${Q_{\dag}(z) \= \hvarphi \circ Q \circ \varphi^{-1}}$, and put
\begin{equation}
  \label{eq:46}
  X_{\dag}
  \=
  \varphi(X_0),
  \varrho_{\dag}
  \=
  r'/r_0,
  \text{ and }
  B_{\dag}
  \=
  \{ z \in \PK \: |z| \ge \varrho_{\dag} \}.
\end{equation}
Then, ${\hvarphi(B_0) = B_{\dag}}$ and ${X_{\dag}}$ is the component of~$Q_{\dag}^{-1}(B_{\dag})$ containing~$\infty$.
So, to prove the theorem it is sufficient to show that~$X_{\dag}$ contains a critical point of~$Q_{\dag}$.
To do this, for each real number~$\varrho$ satisfying ${\varrho < \varrho_{\dag}}$, denote by~$\aX_{\varrho}$ the connected component of~$Q_{\dag}^{-1}(\{x \in \PKber \: |x| > \varrho\})$ containing~$\infty$, and put ${X_{\varrho} \= \aX_{\varrho} \cap \PK}$.
Since~$Q_{\dag}$ has a finite number of critical points and ${\bigcap_{\varrho < \varrho_{\dag}} X_{\varrho} = X_{\dag}}$, it is sufficient to show that, for every real number~$\varrho$ satisfying ${\varrho < \varrho_{\dag}}$, the set~$X_{\varrho}$ contains a critical point of~$Q_{\dag}$.
Fix such~$\varrho$ and suppose that every pole of~$Q_{\dag}$ in~$X_{\varrho}$ is simple, for otherwise there is nothing to prove.

The inequality ${r' \le r_0}$ implies ${\varrho < \varrho_{\dag} \le 1}$.
Let~$G_1$ be the function defined in~\cref{ss:distorded-log-size} with ${R = Q_{\dag}}$ and ${d = 1}$, and put ${D_{\infty} \= \{z \in \PK \: |z| > 1\}}$.
From the above, $D_{\infty}$ contains no critical point of~$Q_{\dag}$,
\begin{equation}
  \label{eq:47}
  Q_{\dag}(0)
  =
  0,
  Q_{\dag}(\infty)
  =
  \infty,
  \max \{ |z| \: z \in Q_{\dag}^{-1}(0) \}
  =
  1,
  \| Q_{\dag}' \|(\infty)
  =
  |Q'(0)|r/r',
\end{equation}
\begin{equation}
  \label{eq:48}
  \wmax(Q_{\dag})
  =
  \wmax(Q),
  \varphi(D)
  =
  \hvarphi(D')
  =
  D_{\infty},
  \text{ and }
  Q_{\dag}(D_{\infty})
  =
  D_{\infty}.
\end{equation}
It follows that ${Q_{\dag}(\xcan) = \xcan}$, that~$Q_{\dag}^{-1}(\xcan)$ is disjoint from~$\{x \in \PKber \: |x| > 1 \}$, and that~$\| Q_{\dag}' \|$ is constant on this set \cite[\emph{Proposition}~3.1]{0FavRiv25}.
Hence,
\begin{equation}
  \label{eq:49}
  \| Q_{\dag}' \|(\xcan)
  =
  \| Q_{\dag}' \|(\infty)
  \text{ and }
  G_1(\xcan)
  =
  -\log \|Q_{\dag}'\|(\infty).
\end{equation}
Let~$x_0$ be in~$\partial \aX_{\varrho}$.
Then, ${Q_{\dag}(x_0) = x(\varrho)}$ and ${\diam(x_0) < 1}$ because ${Q_{\dag}(D_{\infty}) \subseteq D_{\infty}}$.
Combining the last equality in~\eqref{eq:47}, the second equality in~\eqref{eq:49}, and~\eqref{eq:15}, yields
\begin{multline}
  \label{eq:50}
  G_1(x_0)
  =
  \log \varrho + \wf_{Q_{\dag}}(x_0)
  \le
  \log \varrho - \log |\wmax(Q_{\dag})|
  \\ <
  \log \varrho_{\dag} - \log |\wmax(Q_{\dag})|
  =
  -\log \| Q_{\dag}' \|(\infty)
  =
  G_1(\xcan).
\end{multline}
Denote by~$\whx_0$ the point of~$[x_0, \xcan]$ minimizing~$G_1$ that is the closest to~$\xcan$.
So, ${\diam(\whx_0) < 1}$ and ${\partial G_1(\whx_0) > 0}$.
But~$\partial G_1(\whx_0)$ is an integer by \cref{l:distorded-log-size} with ${d = 1}$, so ${\partial G_1(\whx_0) \ge 1}$.
Putting ${L \= \{ \whx_0 \: x_0 \in \partial \aX_{\varrho}\}}$ and applying \cref{l:distorded-log-size} with ${d = 1}$ to each of element of~$L$, yields
\begin{multline}
  \label{eq:51}
  \# L
  \le
  \sum_{\whx \in L} \partial G_1(\whx)
  \\ =
  \sum_{\whx \in L} [2(Z(Q_{\dag}, B(\whx)) - P(Q_{\dag}, B(\whx))) - (Z(Q_{\dag}', B(\whx)) - P(Q_{\dag}', B(\whx)) + 1)].
\end{multline}
For distinct~$x_0$ and~$x_0'$ in~$L$, the point~$\whx_0$ is outside~$[\whx_0', \xcan]$ and~$\whx_0'$ is outside~$[\whx_0, \xcan]$, so ${B(\whx_0) \cap B(\whx_0') = \emptyset}$.
Putting
\begin{equation}
  \label{eq:52}
  \hX
  \=
  \PK \setminus \bigcup_{\whx \in L} B(\whx),
\end{equation}
\begin{equation}
  \label{eq:53}
  2(P(Q_{\dag}, \hX) - Z(Q_{\dag}, \hX)) - (P(Q_{\dag}', \hX) - Z(Q_{\dag}', \hX))
  \ge
  2\# L.
\end{equation}
But ${\hX \subseteq X_{\varrho}}$, so ${Q_{\dag}(\hX) \subseteq \{ z \in \PK \: |z| > \varrho \}}$ and ${Z(Q_{\dag}, \hX) = 0}$.
On the other hand, since ${Q_{\dag}(\infty) = \infty}$ and every pole of~$Q_{\dag}$ in~$X_{\varrho}$ is simple, ${P(Q_{\dag}', \hX) = 2 P(Q_{\dag}, \hX) - 2}$.
Hence, \eqref{eq:53} yields
\begin{equation}
  \label{eq:54}
  Z(Q_{\dag}', \hX)
  \ge
  2\# L - 2.
\end{equation}
Since every point~$\whx$ of~$L$ satisfies ${\diam(\whx) < 1}$, the equality ${\#L = 1}$ would imply
\begin{equation}
  \label{eq:55}
  \{ z \in K \: |z| = 1\}
  \subseteq
  \hX
  \subseteq
  X_{\varrho}
\end{equation}
and hence that~$X_{\varrho}$ contains a zero of~$Q_{\dag}$ by the third equality in~\eqref{eq:47}.
This is absurd, because ${Q_{\dag}(X_{\varrho}) = \{ z \in \PK \: |z| > \varrho\}}$.
This proves ${\# L \ge 2}$ and hence ${Z(Q_{\dag}', \hX) \ge 2}$ by~\eqref{eq:54}.
Since ${\hX \subseteq X_{\varrho}}$, it follows that~$Q_{\dag}$ has a critical point in~$X_{\varrho}$.

\subsection{Proof of \cref{c:critically-mapped'}}
\label{ss:c-proof-critically-mapped'}
If ${Q'(0) = 0}$, then the desired assertion holds with ${v = 0}$.
Suppose ${Q'(0) \neq 0}$.
If~$Q$ has a pole of norm less than or equal to~$\gamma(Q) r_{\bullet}$, then the desired assertion follows from \cref{t:critically-mapped'}.
Suppose the norm of every pole of~$Q$ is strictly larger than~$\gamma(Q) r_{\bullet}$.
This implies that the smallest norm of a zero of~$Q$ different from~$0$ is equal to~$r_{\bullet}$, and that~$Q$ has a power series expansion~$\sum_{n = 1}^{+\infty} a_n z^n$ on~$\{ z \in K \: |z| \le \gamma(Q) r_{\bullet} \}$.
The latter implies ${|a_n|(\gamma(Q) r_{\bullet})^n \to 0}$ as ${n \to +\infty}$ and the former that, for every integer~$n$ satisfying ${n \ge 2}$, the inequality ${|a_n| r_{\bullet}^{n - 1} \le |a_1|}$ holds, with equality for some~$n$.
Denote by~$m$ the largest integer satisfying ${m \ge 2}$ and ${|a_m| r_{\bullet}^{m - 1} = |a_1|}$.
Then,
\begin{equation}
  \label{eq:56}
  \deg_Q(\{z \in K \: |z| \le r_{\bullet}\})
  =
  m
  \text{ and }
  |m|
  \ge
  |\wmax(Q)|
  >
  0.
\end{equation}
Put
\begin{equation}
  \label{eq:57}
  r_0
  \=
  \left( \sup \left\{ |n a_n/a_1|^{\frac{1}{n - 1}} \: n \in \Z, n \ge 2 \right\} \right)^{-1}.
\end{equation}
From the definition of~$m$ and the first inequality in~\eqref{eq:56},
\begin{equation}
  \label{eq:58}
  r_0
  \le
  |m a_m/a_1|^{-\frac{1}{m - 1}}
  =
  |m|^{-\frac{1}{m - 1}} r_{\bullet}
  \le
  \gamma(Q) r_{\bullet}.
\end{equation}
Since ${|a_n|(\gamma(Q) r_{\bullet})^n \to 0}$ as ${n \to +\infty}$, it follows that the supremum in the definition of~$r_0$ is attained.
Combined with an analysis of the \textsc{Newton} polygon of~$Q'$, this implies that~$Q$ has a critical point~$c$ of norm equal to~$r_0$.
Let~$d$ be the largest integer satisfying ${d \ge 2}$ and maximizing~$|a_d| r_0^{d}$.
Then,
\begin{equation}
  \label{eq:59}
  \deg_Q(\{z \in K \: |z| < r_0\})
  =
  d,
  |d|
  \ge
  |\wmax(Q)|
  >
  0,
\end{equation}
and thus, by the definition of~$r_0$ and~\eqref{eq:58},
\begin{multline}
  \label{eq:60}
  |Q(c)|
  \le
  |a_d| r_0^{d}
  \le
  |a_d| r_0^{d - 1} \gamma(Q) r_{\bullet}
  \le
  |a_d| (|d a_d/a_1|)^{-1} \gamma(Q) r_{\bullet}
  \\ =
  |d|^{-1} |Q'(0)| \gamma(Q) r_{\bullet}
  \le
  |\wmax(Q)|^{-1} |Q'(0)| \gamma(Q) r_{\bullet}.
  \qedhere
\end{multline}

\subsection{Proof of \cref{t:strictly-critically-mapped'}}
\label{ss:proof-critically-mapped'}
Suppose~\eqref{eq:39} holds, put
\begin{equation}
  \label{eq:61}
  d
  \=
  \deg_Q(X),
  r_0
  \=
  |\wmax(Q)|^{-(d - 1)} |Q'(0)| r,
  B_0
  \=
  \{ z \in K \: |z| \le r_0 \},
\end{equation}
and denote by~$X_0$ the component of~$Q^{-1}(B_0)$ containing~$0$.
The goal is to prove that~$Q$ has a critical point in~$X_0$ that is not a zero.
To do this, note ${\deg_Q(X_0) \le d}$ and put
\begin{equation}
  \label{eq:62}
  D
  \=
  \{z \in K \: |z| < r\},
  D'
  \=
  Q(D),
  \text{ and }
  r'
  \=
  \diam(D').
\end{equation}
Since~$D'$ has no pole of~$Q$, the set~$D'$ is a ball of~$K$ containing~$0$ and~$r'$ is finite.
If ${r' > r_0}$, then ${B_0 \subseteq D'}$, ${X_0 \subseteq D}$, and thus~$X_0$ is a ball of~$K$ satisfying ${\diam(X_0) < r}$.
On the other hand,
\begin{equation}
  \label{eq:63}
  |Q'(0)|
  =
  \frac{|\wmax(Q)|^{d - 1} r_0}{r}
  <
  \frac{|\wmax(Q)|^{d - 1} r_0}{\diam(X_0)}.
\end{equation}
This yields~\eqref{eq:21} in \cref{l:critically-holomorphic'} with~$r$ replaced by~$\diam(X_0)$ and~$r'$ by~$r_0$.
It follows that~$Q$ has a critical point of norm strictly less than~$\diam(X_0)$ that is not a zero.
Such a critical point belongs to~$X_0$, so the desired assertion holds when ${r' > r_0}$.
For the remainder of the proof, suppose ${r' \le r_0}$, so ${D' \subseteq B_0}$, ${D \subseteq X_0}$, and ${\deg_Q(D) \le d}$.
If ${|Q'(0)| < \exp(-(d - 1)\wf_Q(x(r))) r'/r}$, then \cref{l:critically-holomorphic'} implies that~$Q$ has a critical point in~$D$, and hence in~$X_0$, that is not a zero.
For the remainder of the proof, suppose
\begin{equation}
  \label{eq:64}
  |Q'(0)|
  \ge
  \exp(-(d - 1)\wf_Q(x(r))) r'/r.
\end{equation}

Note that~$r$ and~$r'$ are both in~$|K^{\times}|$.
Let~$\eta$ and~$\eta'$ in~$K$ be such that ${|\eta| = r}$ and ${|\eta'| = r'}$, and let~$\varphi$ and~$\hvarphi$ be the \textsc{M{\"o}bius} transformations given in affine coordinates by
\begin{equation}
  \label{eq:65}
  \varphi(z)
  \=
  \frac{\eta}{z}
  \text{ and }
  \hvarphi(z)
  \=
  \frac{\eta'}{z}.
\end{equation}
Moreover, denote by~$Q_{\dag}$ the rational map with coefficients in~$K$ defined by ${Q_{\dag}(z) \= \hvarphi \circ Q \circ \varphi^{-1}}$, and put
\begin{equation}
  \label{eq:66}
  X_{\dag}
  \=
  \varphi(X_0),
  \varrho_{\dag}
  \=
  r'/r_0,
  \text{ and }
  B_{\dag}
  \=
  \{ z \in \PK \: |z| \ge \varrho_{\dag} \}.
\end{equation}
Then, ${\hvarphi(B_0) = B_{\dag}}$ and ${X_{\dag}}$ is the component of~$Q_{\dag}^{-1}(B_{\dag})$ containing~$\infty$.
So, to prove the theorem it is sufficient to show that~$X_{\dag}$ contains a critical point of~$Q_{\dag}$ that is not a pole.
To do this, for each real number~$\varrho$ satisfying ${\varrho < \varrho_{\dag}}$, denote by~$\aX_{\varrho}$ the connected component of~$Q_{\dag}^{-1}(\{x \in \PKber \: |x| > \varrho\})$ containing~$\infty$, and put ${X_{\varrho} \= \aX_{\varrho} \cap \PK}$.
Since~$Q_{\dag}$ has a finite number of critical points and ${\bigcap_{\varrho < \varrho_{\dag}} X_{\varrho} = X_{\dag}}$, it is sufficient to show that, for every real number~$\varrho$ satisfying ${\varrho < \varrho_{\dag}}$, the set~$X_{\varrho}$ contains a critical point of~$Q_{\dag}$ that is not a pole.
Fix such~$\varrho$ that is sufficiently close to~$\varrho_{\dag}$ so that ${\deg_Q(X_{\varrho}) = d}$.

The inequality ${r' \le r_0}$ implies ${\varrho < \varrho_{\dag} \le 1}$.
Let~$G_{d - 1}$ be the function defined in~\cref{ss:distorded-log-size} with ${R = Q_{\dag}}$ and~$d$ replaced by ${d - 1}$, and put ${D_{\infty} \= \{z \in \PK \: |z| > 1\}}$.
From the above,
\begin{equation}
  \label{eq:67}
  Q_{\dag}(0)
  =
  0,
  Q_{\dag}(\infty)
  =
  \infty,
  \max \{ |z| \: z \in Q_{\dag}^{-1}(0) \}
  =
  1,
  \| Q_{\dag}' \|(\infty)
  =
  |Q'(0)|r/r',
\end{equation}
\begin{equation}
  \label{eq:68}
  \wmax(Q_{\dag})
  =
  \wmax(Q),
  \varphi(D)
  =
  \hvarphi(D')
  =
  D_{\infty},
  \text{ and }
  Q_{\dag}(D_{\infty})
  =
  D_{\infty}.
\end{equation}
It follows that ${Q_{\dag}(\xcan) = \xcan}$ and, by~\eqref{eq:64},
\begin{equation}
  \label{eq:69}
  G_{d - 1}(\xcan)
  =
  (d - 1) \wf_R(\xcan)
  \ge
  -\log (\| Q_{\dag}' \|(\infty)).
\end{equation}
Let~$x_0$ be in~$\partial \aX_{\varrho}$.
Then, ${Q_{\dag}(x_0) = x(\varrho)}$ and ${\diam(x_0) < 1}$ because ${Q_{\dag}(D_{\infty}) \subseteq D_{\infty}}$.
Combining the last equality in~\eqref{eq:67}, \eqref{eq:69}, and \eqref{eq:15}, yields
\begin{multline}
  \label{eq:70}
  G_{d - 1}(x_0)
  =
  \log \varrho + (d - 1)\wf_{Q_{\dag}}(x_0)
  \le
  \log \varrho - (d - 1) \log |\wmax(Q_{\dag})|
  \\ <
  \log \varrho_{\dag} - (d - 1) \log |\wmax(Q_{\dag})|
  =
  -\log \| Q_{\dag}' \|(\infty)
  \le
  G_{d - 1}(\xcan).
\end{multline}
Denote by~$\whx_0$ the point of~$[x_0, \xcan]$ minimizing~$G_{d - 1}$ that is the closest to~$\xcan$.
So, ${\diam(\whx_0) < 1}$ and ${\partial G_{d - 1}(\whx_0) > 0}$.
But~$\partial G_{d - 1}(\whx_0)$ is an integer by \cref{l:distorded-log-size} with~$d$ replaced by ${d - 1}$, so ${\partial G_{d - 1}(\whx_0) \ge 1}$.
Putting ${L \= \{ \whx_0 \: x_0 \in \partial \aX_{\varrho}\}}$ and applying \cref{l:distorded-log-size} with~$d$ replaced by ${d - 1}$ to each of element of~$L$, yields
\begin{multline}
  \label{eq:71}
  \# L
  \le
  \sum_{\whx \in L} \partial G_{d - 1}(\whx)
  \\ =
  \sum_{\whx \in L} [d(Z(Q_{\dag}, B(\whx)) - P(Q_{\dag}, B(\whx))) - (d - 1)(Z(Q_{\dag}', B(\whx)) - P(Q_{\dag}', B(\whx)) + 1)].
\end{multline}
For distinct~$x_0$ and~$x_0'$ in~$L$, the point~$\whx_0$ is outside~$[\whx_0', \xcan]$ and~$\whx_0'$ is outside~$[\whx_0, \xcan]$, so ${B(\whx_0) \cap B(\whx_0') = \emptyset}$.
Putting
\begin{equation}
  \label{eq:72}
  \hX
  \=
  \PK \setminus \bigcup_{\whx \in L} B(\whx),
\end{equation}
the right-most term in~\eqref{eq:71} is equal to
\begin{equation}
  \label{eq:73}
  d(P(Q_{\dag}, \hX) - Z(Q_{\dag}, \hX)) - (d - 1)(P(Q_{\dag}', \hX) - Z(Q_{\dag}', \hX) + \# L).
\end{equation}
But ${\hX \subseteq X_{\varrho}}$, so ${P(Q_{\dag}, \hX) \le d}$, ${Q_{\dag}(\hX) \subseteq \{ z \in \PK \: |z| > \varrho \}}$, ${Z(Q_{\dag}, \hX) = 0}$, and thus~\eqref{eq:71} yields
\begin{multline}
  \label{eq:74}
  (d - 1)Z(Q_{\dag}', \hX)
  \ge
  \# L - P(Q_{\dag}, \hX) + (d - 1) (P(Q_{\dag}', \hX) - P(Q_{\dag}, \hX) + \# L)
  \\ \ge
  \# L - 1 + (d - 1) (P(Q_{\dag}', \hX) - P(Q_{\dag}, \hX) + \# L - 1).
\end{multline}
On the other hand, from ${Q_{\dag}(\infty) = \infty}$ and ${\deg_{Q_{\dag}}(\infty) = 1}$ it follows that the poles of~$Q_{\dag}'$ coincide with the finite poles of~$Q_{\dag}$.
Together with ${|\wmax(Q_{\dag})| > 0}$, this implies that every pole~$z_0$ of~$Q_{\dag}'$ satisfies
\begin{equation}
  \label{eq:75}
  p
  \not\mid
  \deg_{Q_{\dag}}(z_0)
  \text{ and }
  \ord_{Q_{\dag}'}(z_0)
  =
  \deg_{Q_{\dag}}(z_0) + 1.
\end{equation}
Thus,
\begin{equation}
  \label{eq:76}
  P(Q_{\dag}', \hX)
  =
  P(Q_{\dag}, \hX) + \#(\hX \cap Q_{\dag}^{-1}(\infty)) - 2
\end{equation}
and~\eqref{eq:74} yields
\begin{equation}
  \label{eq:77}
  (d - 1) Z(Q_{\dag}', \hX)
  \ge
  \# L - 1 + (d - 1) (\#(\hX \cap Q_{\dag}^{-1}(\infty)) + \# L - 3).
\end{equation}
Since every point~$\whx$ of~$L$ satisfies ${\diam(\whx) < 1}$, the equality ${\#L = 1}$ would imply
\begin{equation}
  \label{eq:78}
  \{ z \in K \: |z| = 1\}
  \subseteq
  \hX
  \subseteq
  X_{\varrho}
\end{equation}
and hence that~$X_{\varrho}$ contains a zero of~$Q_{\dag}$ by the third equality in~\eqref{eq:67}.
This is absurd, because ${Q_{\dag}(X_{\varrho}) = \{ z \in \PK \: |z| > \varrho\}}$.
This proves ${\# L \ge 2}$.
Combined with ${\#(\hX \cap Q_{\dag}^{-1}(\infty)) \ge 1}$ and~\eqref{eq:77}, this implies ${(d - 1)Z(Q_{\dag}', \hX) \ge 1}$ and hence ${d \ge 2}$ and ${Z(Q_{\dag}', \hX) \ge 1}$.
Since ${\hX \subseteq X_{\dag}}$, it follows that~$Q_{\dag}$ has a critical point in~$X_{\dag}$ that is not a pole.

\subsection{Proof of \cref{t:strictly-critically-mapped''}}
\label{ss:proof-critically-mapped'''}
Suppose~\eqref{eq:40} holds, put ${d \= \deg_Q(X)}$, and for each real number~$\check{r}$ satisfying ${\check{r} > \whr_{\bullet}}$, denote by~$X_{\check{r}}$ the component of~$Q^{-1}(\{ z \in K \: |z| < \check{r} \})$ containing~$0$.
Since~$Q$ has a finite number of critical points, it is sufficient to prove that for every such~$\check{r}$ the set~$X_{\check{r}}$ contains a critical point of~$Q$ that is not a zero.
Fix a real number~$\check{r}$ satisfying ${\check{r} > \whr_{\bullet}}$ that is sufficiently close to~$\whr_{\bullet}$, so that ${\deg_Q(X_{\check{r}}) = d}$.

Let~$\eta$ in~$K$ be such that ${|\eta| = r_{\bullet}}$ and let~$\varphi$ and~$\hvarphi$ be the \textsc{M{\"o}bius} transformations given in affine coordinates by
\begin{equation}
  \label{eq:79}
  \varphi(z)
  \=
  \frac{\eta}{z}
  \text{ and }
  \hvarphi(z)
  \=
  \frac{\eta Q'(0)}{z}.
\end{equation}
Moreover, let~$Q_{\dag}$ be the rational map with coefficients in~$K$ defined by ${Q_{\dag}(z) \= \hvarphi \circ Q \circ \varphi^{-1}}$, and put
\begin{equation}
  \label{eq:80}
  X_{\dag}
  \=
  \varphi(X_{\check{r}}),
  \varrho_{\dag}
  \=
  \frac{|Q'(0)| r_{\bullet}}{\check{r}},
  \text{ and }
  D_{\dag}
  \=
  \{x \in \PK \: |x| > \varrho_{\dag} \}.
\end{equation}
Then, ${\hvarphi(\{ x \in \PKber \: |x| < \check{r} \}) = D_{\dag}}$ and~$X_{\dag}$ is the component of~$Q_{\dag}^{-1}(D_{\dag})$ containing~$\infty$.
Thus, to prove the theorem it is sufficient to prove that~$X_{\dag}$ contains a critical point of~$Q_{\dag}$ that is not a pole.

Note that
\begin{equation}
  \label{eq:81}
  Q_{\dag}(0)
  =
  0,
  Q_{\dag}(\infty)
  =
  \infty,
  \deg_{Q_{\dag}}(\infty)
  =
  1,
  \| Q_{\dag}' \|(\infty)
  =
  1,
\end{equation}
\begin{equation}
  \label{eq:82}
  \max \{ |z| \: z \text{ finite zero or pole of~$Q_{\dag}$} \}
  =
  1,
\end{equation}
\begin{equation}
  \label{eq:83}
  0
  \not\in
  X_{\dag},
  \deg_{Q_{\dag}}(X_{\dag})
  =
  d,
  \text{ and }
  0
  <
  \varrho_{\dag}
  <
  |\wmax(Q_{\dag})|^d.
\end{equation}
Let~$G_d$ be the function defined in~\cref{ss:distorded-log-size} with ${R = Q_{\dag}}$ and put ${D_{\infty} \= \{ z \in \PK \: |z| > 1 \}}$.
Since~$D_{\infty}$ contains no zero of~$Q_{\dag}$ and ${Q_{\dag}(\infty) = \infty}$, the set~$Q_{\dag}(D_{\infty})$ is a disk of~$\PK$ containing~$\infty$ but not~$0$.
Moreover, $Q_{\dag}$ is univalent on~$D_{\infty}$ because~$\infty$ is the only preimage of~$\infty$ by~$Q_{\dag}$ in~$D_{\infty}$ and ${\deg_{Q_{\dag}}(\infty) = 1}$.
Combined with ${\| Q_{\dag}' \|(\infty) = 1}$, this implies that~$Q_{\dag}$ maps~$D_{\infty}$ univalently onto itself.
In particular, ${Q_{\dag}(\xcan) = \xcan}$ and ${\wdeg{Q_{\dag}}(\xcan) = 1}$.
Thus, ${\wf_R(\xcan) = 0}$ by~\eqref{eq:15} and hence ${G_d(\xcan) = 0}$.
Let~$\aX_{\dag}$ be the fundamental open set of~$\PKber$ satisfying ${\aX_{\dag} \cap \PK = X_{\dag}}$.
Let~$x_0$ be in~$\partial \aX_{\dag}$.
Then, $Q_{\dag}(x_0)$ is the point of~$\PKber$ associated with ${\{z \in K \: |z| \le \varrho_{\dag}\}}$.
Thus, \eqref{eq:83} and~\eqref{eq:15} yield
\begin{equation}
  \label{eq:84}
  G_d(x_0)
  =
  \log \varrho_{\dag} - d\wf_{Q_{\dag}}(x_0)
  \le
  \log \varrho_{\dag} - d\log |\wmax(Q_{\dag})|
  <
  0
  =
  G_d(\xcan).
\end{equation}
Denote by~$\whx_0$ the point of~$[x_0, \xcan]$ minimizing~$G_d$ that is the closest to~$\xcan$.
So, ${\diam(\whx_0) < 1}$ and ${\partial G_d(\whx_0) > 0}$.
But~$\partial G_d(\whx_0)$ is an integer by \cref{l:distorded-log-size}, so ${\partial G_d(\whx_0) \ge 1}$.
Putting ${L \= \{ \whx_0 \: x_0 \in \partial \aX_{\dag}\}}$ and applying \cref{l:distorded-log-size} to each of element of~$L$, yields
\begin{multline}
  \label{eq:85}
  \# L
  \le
  \sum_{\whx \in L} \partial G_d(\whx)
  \\ =
  \sum_{\whx \in L} [(d + 1)(Z(Q_{\dag}, B(\whx)) - P(Q_{\dag}, B(\whx))) - d (Z(Q_{\dag}', B(\whx)) - P(Q_{\dag}', B(\whx)) + 1)].
\end{multline}
For distinct~$x_0$ and~$x_0'$ in~$L$, the point~$\whx_0$ is outside~$[\whx_0', \xcan]$ and~$\whx_0'$ is outside~$[\whx_0, \xcan]$, so ${B(\whx_0) \cap B(\whx_0') = \emptyset}$.
Putting
\begin{equation}
  \label{eq:86}
  \hX
  \=
  \PK \setminus \bigcup_{\whx \in L} B(\whx),
\end{equation}
the right-most term in~\eqref{eq:85} is equal to
\begin{equation}
  \label{eq:87}
  (d + 1)(P(Q_{\dag}, \hX) - Z(Q_{\dag}, \hX)) - d (P(Q_{\dag}', \hX) - Z(Q_{\dag}', \hX) + \# L).
\end{equation}
But ${\hX \subseteq X_{\dag}}$, so ${P(Q_{\dag}, \hX) \le d}$, ${Q_{\dag}(\hX) \subseteq D_{\dag}}$, ${Z(Q_{\dag}, \hX) = 0}$, and thus~\eqref{eq:85} yields
\begin{multline}
  \label{eq:88}
  dZ(Q_{\dag}', \hX)
  \ge
  \# L - P(Q_{\dag}, \hX) + d (P(Q_{\dag}', \hX) - P(Q_{\dag}, \hX) + \# L)
  \\ \ge
  \# L + d (P(Q_{\dag}', \hX) - P(Q_{\dag}, \hX) + \# L - 1).
\end{multline}
On the other hand, from ${Q_{\dag}(\infty) = \infty}$ and ${\deg_{Q_{\dag}}(\infty) = 1}$ it follows that the poles of~$Q_{\dag}'$ coincide with the finite poles of~$Q_{\dag}$.
Together with the hypothesis ${|\wmax(Q_{\dag})| > 0}$, this implies that every pole~$z_0$ of~$Q_{\dag}'$ satisfies
\begin{equation}
  \label{eq:89}
  \ord_{Q_{\dag}'}(z_0)
  =
  \deg_{Q_{\dag}}(z_0) + 1.
\end{equation}
Thus,
\begin{equation}
  \label{eq:90}
  P(Q_{\dag}', \hX)
  =
  P(Q_{\dag}, \hX) + \#(\hX \cap Q_{\dag}^{-1}(\infty)) - 2
\end{equation}
and~\eqref{eq:88} yields
\begin{equation}
  \label{eq:91}
  dZ(Q_{\dag}', \hX)
  \ge
  \# L + d(\#(\hX \cap Q_{\dag}^{-1}(\infty)) + \# L - 3).
\end{equation}
The next step is to prove
\begin{equation}
  \label{eq:92}
  \#(\hX \cap Q_{\dag}^{-1}(\infty)) + \# L
  \ge
  3.
\end{equation}
Since ${\#(\hX \cap Q_{\dag}^{-1}(\infty)) \ge 1}$, the inequality~\eqref{eq:92} holds if ${\# L \ge 2}$.
Suppose ${\#L = 1}$.
Then, $\hX$ is a disk.
Since~$0$ is outside~$X_{\dag}$ and thus~$\hX$, and the unique point~$\whx$ in~$L$ satisfies ${\diam(\whx) < 1}$, the set~${\{ z \in K \: |z| = 1 \}}$ is contained in~$\hX$.
This implies that~${\{ z \in K \: |z| = 1 \}}$ contains no zero of~$Q_{\dag}$ and, by~\eqref{eq:82}, that it contains a pole of~$Q_{\dag}$.
Thus, ${\#(\hX \cap Q_{\dag}^{-1}(\infty)) \ge 2}$ and~\eqref{eq:92} holds in all of the cases.
Combining~\eqref{eq:91} and~\eqref{eq:92} yields ${d Z(Q_{\dag}', \hX) \ge 1}$ and hence ${Z(Q_{\dag}', \hX) \ge 1}$.
Since ${\hX \subseteq X_{\dag}}$, it follows that~$Q_{\dag}$ has a critical point in~$X_{\dag}$ that is not a pole.

\section{Locating critical points in immediate basins}
\label{s:critically-attracted}

This section states and proves refined versions of Theorems~\ref{t:critically-attracted} and~\ref{t:infinitely-attracted} and \cref{c:critical-bound} from~\cref{s:introdcution}.
Throughout this section, fix a rational map~$R$ of degree~$d$ at least two with coefficients in~$K$ and an attracting cycle~$\cO$ of~$R$.
The multiplier~$\lambda$ and \textsc{Lyapunov} exponent~$\chi$ of~$\cO$, are given by
\begin{equation}
  \label{eq:93}
  \lambda
  \=
  \prod_{z \in \cO} \| R' \|(z)
  \text{ and }
  \chi
  \=
  \frac{1}{\# \cO} \log |\lambda|.
\end{equation}
The \emph{basin of attraction of~$\cO$}, is the set of points~$x$ of~$\PKber$ such that~$(R^{n \# \cO}(x))_{n = 1}^{+\infty}$ converges to a point in~$\cO$.
For each~$z$ in~$\cO$, denote by~$\aA(z)$ the connected component of the basin of attraction of~$\cO$ that contains~$z$.
The \emph{immediate basin of~$\cO$}, is~$\bigcup_{z \in \cO} \aA(z)$.

Recall from~\cref{s:preliminaries} that the maximal inseparability degree~$\wmax(R)$ of~$R$, is defined by
\begin{equation}
  \label{eq:94}
  \wmax(R)
  =
  \max \{ \wdeg{R}(x) \: x \in \HK \}.
\end{equation}
Noting that~$(\wmax(R^n))_{n = 1}^{+\infty}$ is sub-multiplicative, define
\begin{equation}
  \label{eq:95}
  \delta(R)
  \=
  \lim_{n \to +\infty} (\wmax(R^n))^{\frac{1}{n}}.
\end{equation}
Since ${\lambda(d) \le |\wmax(R)| \le |\delta(R)|}$, \cref{t:critically-attracted} in~\cref{s:introdcution} is a direct consequence of the following theorem.

\begin{custtheo}{A'}
  \label{t:critically-attracted'}
  If ${\chi < \log |\delta(R)|}$, then the immediate basin of~$\cO$ contains a critical point~$c$ of~$R$ in~$\PK$, together with a disk~$D$ that includes~$R(c)$, intersects~$\cO$, and satisfies ${R^{\#\cO}(D) \subseteq D}$.
\end{custtheo}

The following corollary is a direct consequence of the previous theorem.

\begin{coro}
  \label{c:critical-bound'}
  The number of cycles of~$R$ whose \textsc{Lyapunov} exponent is strictly less than~$\log |\delta(d)|$ is finite, less than or equal to the number of critical points of~$R$ in~$\PK$.
\end{coro}

The following theorem relies on some notation and terminology.
Denote by~$\cO_{\crit}$ the set of those~$z$ in~$\cO$ such that~$R$ has a critical point in ${\aA(z) \cap \PK}$.
Note that, for each~$z$ in~$\cO$, the degree~$\deg_R(\aA(z))$ is defined because~$\aA(z)$ is the connected component of~$R^{-1}(\aA(R(z)))$ containing~$z$.
The immediate basin of~$\cO$ is \emph{of \textsc{Cantor} type}, if, for each~$z$ in~$\cO$, the boundary of~$\aA(z)$ is a \textsc{Cantor} set.
The immediate basin of~$\cO$ is not of \textsc{Cantor} type, then, for each~$z$ in~$\cO$, the set~$\aA(z)$ is a disk \cite[\emph{Th{\'e}or{\`e}me}~2]{Riv03c}.

The next notation depends on whether the immediate basin of~$\cO$ is of \textsc{Cantor} type or not.
If it is, then put, for each~$z$ in~$\cO$,
\begin{align}
  \label{eq:96}
  \ell(z)
  & \=
    \begin{cases}
      |\wmax(R)|^{\deg_R(\aA(z))}
      & \text{if } z \in \cO_{\crit};
      \\
      |\wmax(R)|
      & \text{if } z \not\in \cO_{\crit}.
    \end{cases}
\end{align}
Suppose that the immediate basin of~$\cO$ is not of \textsc{Cantor} type.
If for each~$z'$ in ${\cO \setminus \cO_{\crit}}$ the equality ${|\deg_R(\aA(z'))| = 1}$ holds, then put ${\theta \= 1}$.
Otherwise, ${p}$ is a prime number and then put ${\theta \= |p|^{\frac{1}{p - 1}}}$.
In all of the cases, put, for each~$z$ in~$\cO$,
\begin{align}
  \label{eq:97}
  \ell(z)
  & \=
  |\deg_R(\aA(z))|
\intertext{if~$z$ is outside~$\cO_{\crit}$ and, otherwise,}
  \label{eq:98}
  \hell(z)
  & \=
  \min \{ |\deg_R(D)|^{\deg_R(D)} \: D \text{ disk of~$\PK$ satisfying } z \in D \subseteq \aA(z) \}
  \intertext{and}
  \label{eq:99}
  \ell(z)
  & \=
  \min \{ \theta, \hell(z) \}.
\end{align}

In all of the cases, put
\begin{equation}
  \label{eq:100}
  \ell(\cO)
  \=
  \left( \prod_{z \in \cO} \ell(z) \right)^{1 / \# \cO}.
\end{equation}
The inequality
\begin{align}
  \label{eq:101}
  \ell(\cO)
  & \ge
    |\wmax(R)|^{1 + (d - 1) \# \cO_{\crit} / \# \cO}
    \intertext{holds because, for every~$z$ in~$\cO$,}
    \label{eq:102}
    \ell(z)
  & \ge
    \begin{cases}
      |\wmax(R)|^d
      & \text{if } z \in \cO_{\crit};
      \\
      |\wmax(R)|
      & \text{if } z \not\in \cO_{\crit}.
    \end{cases}
\end{align}
Since ${\lambda(d) \le |\wmax(R)|}$ and ${\# \cO_{\crit} \le 2d - 2}$, \cref{t:infinitely-attracted} in~\cref{s:introdcution} is a direct consequence of the following theorem.

\begin{custtheo}{B'}
  \label{t:infinitely-attracted'}
  If ${-\infty < \chi < \log \ell(\cO)}$, then the immediate basin of~$\cO$ contains a critical point~$c$ of~$R$ in~$\PK$ whose forward orbit is infinite, together with a disk~$D$ containing~$R(c)$, intersecting~$\cO$, and such that~$R^{\#\cO}$ maps~$D$ univalently into itself.
\end{custtheo}

When~$\cO$ consists of a fixed point, \cref{t:infinitely-attracted'} implies \cite[Theorem~4.1]{BenIngJonLev14}.
\cref{c:infinitely-attracted'} is a direct consequence of \cref{t:infinitely-attracted'} and the inequality ${\# \cO_{\crit} \le d - 1}$ for polynomials.
The following corollary is a direct consequence of \cref{t:infinitely-attracted'} and the inequality ${\# \cO_{\crit} \le 2d - 2}$.

\begin{coro}
  \label{c:infinitely-attracted''}
  Let~$R$ be a rational map of degree~$d$ at least two with coefficients in~$K$, and~$\cO$ an attracting cycle of~$R$ of \textsc{Lyapunov} exponent~$\chi$ whose immediate basin is not of \textsc{Cantor} type.
  Then,
  \begin{equation}
    \label{eq:103}
    -\infty
    <
    \chi
    <
    \begin{cases}
      \log \hlambda(d)
      & \text{if } \# \cO \le 2d - 2;
      \\
      \log \lambda(d) + \frac{2d - 2}{\# \cO} \log \frac{\hlambda(d)}{\lambda(d)}
      & \text{if } \# \cO > 2d - 2
    \end{cases}
  \end{equation}
  implies that~$\cO$ attracts a critical point of~$R$ in~$\PK$ whose forward orbit is infinite.
\end{coro}

The proofs of Theorems~\ref{t:critically-attracted'} and~\ref{t:infinitely-attracted'} occupy \S\S\ref{ss:proof-critically-attracted'} and~\ref{ss:proof-infinitely-attracted'}, respectively.

\subsection{Proof of \cref{t:critically-attracted'}}
\label{ss:proof-critically-attracted'}
For each~$z$ in~$\cO$, \cref{l:disk-basin} provides the main properties of the largest disk~$D(z)$ containing~$z$ and contained in the immediate basin of~$\cO$.
Building on~\cref{l:critically-holomorphic} and \cref{t:critically-mapped'}, \cref{l:critical-criterion} gives a criterion ensuring that~$D(R(z))$ contains a critical value.
The proof of \cref{t:critically-attracted'} verifies that this criterion is satisfied for at least one point of~$\cO$.

The rational map~$R$ has a fixed point in ${\PK \setminus \cO}$.
Changing coordinates if necessary, throughout the rest of this section suppose ${\cO \subset K}$ and ${R(\infty) = \infty}$.

\begin{lemm}
  \label{l:disk-basin}
  For each~$z$ in~$\cO$, let~$D(z)$ be the union of all the balls of~$\PK$ containing~$z$ and contained in~$\aA(z)$.
  Then, $D(z)$ is a rational open disk of~$\PK$.
  Furthermore, ${R(D(z)) \subseteq D(R(z))}$, and, if this inclusion is strict, then~$R$ has a pole~$\pi$ satisfying
  \begin{equation}
    \label{eq:104}
    |\pi - z| = \diam(D(z)).
  \end{equation}
\end{lemm}

\begin{proof}
  To prove the first assertion, put ${R_0 \= R^{\# \cO}}$ and let~$z_0$ be in~$\cO$.
  After an affine coordinate change if necessary, suppose ${z_0 = 0}$ and put
  \begin{equation}
    \label{eq:105}
    r_0
    \=
    \max \{ |z| \: z \in R_0^{-1}(\infty) \}.
  \end{equation}
  Furthermore, put ${x_0 \= \infty}$ if ${r_0 = +\infty}$ and denote by~$x_0$ the point of~$\PKber$ associated with the ball ${\{z \in K \: |z| \le r_0\}}$ otherwise.
  Then, for every~$x$ in~$]0, x_0]$, the point~$R_0(x)$ belongs to~$]0, \infty]$ and
  \begin{equation}
    \label{eq:106}
    R_0(\{ z \in \PK \: |z| < |x| \})
    =
    \{ z \in \PK \: |z| < |R_0(x)| \}.
  \end{equation}
  Suppose that~$R_0(x_0)$ belongs to~$]0, x_0[$ and put ${D_0 \= \{z \in K \: |z| < r_0\}}$.
  Then, $r_0$ is finite and~$D_0$ is contained in~$\aA(z_0)$ by~\eqref{eq:106} and \textsc{Schwarz}' Lemma \cite[\S1.3.1]{Riv03c}.
  Furthermore, every ball of~$\PK$ containing~$D_0$ strictly, intersects~$R_0^{-1}(\infty)$ and it is thus not contained in~$\aA(z_0)$.
  It follows that ${D(z_0) = D_0}$ and hence that~$D(z_0)$ is a rational open disk of~$\PK$.
  It remains to consider the case where~$R_0(x_0)$ is outside~$]0, x_0[$.
  Since~$0$ is an attracting fixed point of~$R_0$, if~$r_0$ is finite, then~$R_0$ fixes a point of~$]0, x_0]$.
  On the other hand, if ${r_0 = +\infty}$, then ${R_0^{-1}(\infty) = \{ \infty \}}$, the fixed point~$\infty$ of~$R_0$ is attracting, and hence~$R_0$ fixes a point of~$]0, \infty[$.
  In both cases, $R_0$ fixes a point of~$]0, x_0]$ in~$\HK$.
  Denote by~$x_*$ the fixed point of~$R_0$ in~$]0, x_0]$ that is closest to~$0$ and put ${D_* \= \{z \in K \: |z| < |x_*|\}}$.
  Then, $x_*$ is of type~II and~$D_*$ is contained in~$\aA(z_0)$ by~\eqref{eq:106} and \textsc{Schwarz}' Lemma \cite[\S1.3.1]{Riv03c}.
  Furthermore, every ball of~$\PKber$ containing~$D_*$ strictly, contains~$x_*$ and hence it is not contained in~$\aA(z_0)$.
  It follows that ${D_* = D(z_0)}$ and that~$D(z_0)$ is an rational open disk of~$\PK$.
  The proof of the first assertion is thus complete.

  To prove the second assertion, suppose ${R(D(z_0)) \neq D(R(z_0))}$ and put ${r_0 \= \diam(D(z_0))}$.
  Since~$D(z_0)$ is a disk without poles of~$R$, the set~$R(D(z_0))$ is a ball of~$\PKber$ containing~$R(z_0)$ and contained in~$\aA(R(z_0))$.
  It follows that~$R(D(z_0))$ is contained in~$D(R(z_0))$.
  Together with the hypothesis ${R(D(z_0)) \neq D(R(z_0))}$, this implies that~$R(D(z_0))$ is strictly contained in~$D(R(z_0))$.
  If~$R$ had no pole~$\pi$ satisfying ${|\pi - z_0| = r_0}$, then there would be~$\varrho_0$ in~$]r_0, +\infty[$ such that
  \begin{equation}
    \label{eq:107}
    R(\{z \in K \: |z - z_0| < \varrho_0\})
    \subseteq
    D(R(z_0)).
  \end{equation}
  This would imply
  \begin{equation}
    \label{eq:108}
    \{z \in K \: |z - z_0| < \varrho_0\})
    \subseteq
    \aA(z_0),
  \end{equation}
  and would contradict the maximality of~$D(z_0)$.
\end{proof}

For each~$z$ in~$\cO$, let~$D(z)$ be the rational open disk of~$\PK$ given by \cref{l:disk-basin} and~$X(z)$ the component of~$R^{-1}(D(R(z)))$ containing~$z$.
Then,
\begin{equation}
  \label{eq:109}
  D(z)
  \subseteq
  X(z)
  \subseteq
  \aA(z)
  \text{ and }
  R^{\#\cO}(D(z))
  \subseteq
  D(z).
\end{equation}

\begin{lemm}
  \label{l:critical-criterion}
  Suppose~$z$ in~$\cO$ satisfies
  \begin{equation}
    \label{eq:110}
    |R'(z)|
    <
    |\wmax(R)| \frac{\diam(D(R(z)))}{\diam(D(z))}.
  \end{equation}
  Then, $R$ has a critical point in~$X(z)$.
\end{lemm}

\begin{proof}
  If ${R(D(z)) = D(R(z))}$, then ${X(z) = D(z)}$ and \cref{l:critically-holomorphic} yields the desired assertion.
  If~$R(D(z))$ is strictly contained in~$D(R(z))$, then~$R$ has a pole satisfying ${|\pi - z| = \diam(D(z))}$ by \cref{l:disk-basin}, and \cref{t:critically-mapped'} with ${Q(z) = R(z + z) - R(z)}$ yields the desired assertion.
\end{proof}

\begin{proof}[Proof of \cref{t:critically-attracted'}]
  Replacing~$R$ by an iterate if necessary, suppose ${\chi < \log |\wmax(R)|}$.
  For each~$z$ in~$\cO$, put ${r(z) \= \diam(D(z))}$.
  Then,
  \begin{equation}
    \label{eq:111}
    \prod_{z \in \cO} \left( |R'(z)| r(z)/r(R(z)) \right)
    =
    |\lambda|
    =
    \exp(\#\cO \times \chi)
    <
    |\wmax(R)|^{\#\cO},
  \end{equation}
  and hence that there is~$z_0$ in~$\cO$ satisfying
  \begin{equation}
    \label{eq:112}
    |R'(z_0)|
    <
    |\wmax(R)| r(R(z_0))/r(z_0).
  \end{equation}
  Together with \cref{l:critical-criterion}, this implies that~$R$ has a critical point in~$X(z_0)$.
  The desired assertion follows with ${D = D(z_0)}$.
\end{proof}

\subsection{Proof of \cref{t:infinitely-attracted'}}
\label{ss:proof-infinitely-attracted'}
As in~\cref{ss:proof-critically-attracted'}, suppose ${\cO \subset K}$ and ${R(\infty) = \infty}$.
For each~$z$ in~$\cO$, let~$D(z)$ the rational open disk of~$\PK$ given by \cref{l:disk-basin} and~$X(z)$ the component of~$R^{-1}(D(R(z)))$ containing~$z$.
If the immediate basin of~$\cO$ is not of \textsc{Cantor} type, then each of the inclusions in~\eqref{eq:109} is an equality.
In particular, ${R(D(z)) = D(R(z))}$.

The first step in the proof of \cref{t:infinitely-attracted'} is to prove that~$\cO_{\crit}$ is nonempty using \cref{l:critically-holomorphic} and \cref{l:critical-criterion}, a specific version of \cref{t:critically-mapped'}.
The main step is to prove that, for a carefully chosen~$z_0$ in~$\cO$, the following holds for every~$n$ in~$\Z_{> 0}$: There is~$j_n$ in~$\{0, \ldots, n - 1\}$ and a critical value~$v_n$ of~$R$ in~$K$ that is close to, but distinct from, $R^{j_n + 1}(z_0)$ and such that~$R^{n - j_n - 1}$ is univalent on~$\{ z \in K \: |z - R^{j_n + 1}(z_0)| \le |v_n - R^{j_n + 1}(z_0)|\}$.
The base case ${n = 1}$ follows from \cref{t:strictly-critically-mapped'} and \cref{l:critically-holomorphic'}.
The inductive step follows from \cref{t:strictly-critically-mapped''} if~$z_n$ belongs to~$\cO_{\crit}$, and from the general lemma below otherwise.
\cref{t:infinitely-attracted'} then follows from the main step and the fact that the number of critical points of~$R$ in~$\PK$ is finite.

The proof of \cref{t:infinitely-attracted'} relies on the following general lemma.

\begin{lemm}
  \label{l:injectivity-radius}
  Let~$Q$ be a nonconstant rational map with coefficients in~$K$, let~$r$ in~$\R_{> 0}$ be such that~$Q$ has no pole or critical point in~$\{ z \in K \: |z| < r \}$, and put ${d \= \deg_Q(\{ z \in K \: |z| < r \})}$.
  Then, ${Q}$ is univalent on ${\{ z \in K \: |z| < r\}}$ if ${|d| = 1}$, and on ${\{ z \in K \: |z| < |p|^{\frac{1}{p - 1}} r\}}$ if ${0 < |d| < 1}$.
\end{lemm}

\begin{proof}
  Put ${\theta \= 1}$ if ${|d| = 1}$ and ${\theta \=|p|^{\frac{1}{p - 1}}}$ if ${0 < |d| < 1}$, and let~$w$ and~$w'$ in~$\{ z \in K \: |z| < \theta r \}$ be distinct.
  Since~${Q(z + w) - Q(w)}$ has no pole in~$\{ z \in K \: |z| < r \}$, it has a power series expansion~$\sum_{n = 1}^{+\infty} a_n z^n$ on this set.
  Since this map has no critical point in ${\{ z \in K \: |z| < r \}}$, an analysis of the \textsc{Newton} polygon of its derivative reveals that, for every~$n$ in~$\Z_{> 0}$, the inequality ${|na_n| r^{n - 1} \le |a_1|}$ holds.
  On the other hand, from the definition of~$d$, for every~$n$ in~$\Z_{> 0}$ the inequality ${|a_n| r^n \le |a_d| r^d}$ holds.
  Thus, ${|d| = 1}$ implies ${|a_d| r^{d - 1} = |a_1|}$ and, for every~$n$ in~$\Z_{> 0}$,
  \begin{equation}
    \label{eq:113}
    |a_n| r^{n - 1}
    \le
    |a_d| r^{d - 1}
    =
    |a_1|.
  \end{equation}
  On the other hand, ${0 < |d| < 1}$ implies ${0 < |p| < 1}$ and that, for every~$n$ in~$\Z_{> 0}$,
  \begin{equation}
    \label{eq:114}
    |a_n| (\theta r)^{n - 1}
    =
    |p|^{\frac{n - 1}{p - 1}} |a_n| r^{n - 1}
    \le
    |na_n| r^{n - 1}
    \le
    |a_1|.
  \end{equation}
  In all of the cases, for every~$n$ in~$\Z_{> 0}$, the inequality ${|a_n| r^{n - 1} \le |a_1|}$ holds and thus
  \begin{equation}
    \label{eq:115}
    |Q(w') - Q(w) - a_1(w' - w)|
    \le
    \max \{ |a_n| \times |w' - w|^n \: n \ge 2 \}
    <
    |a_1| \times |w' - w|.
  \end{equation}
  Hence
  \begin{equation}
    \label{eq:116}
    |Q(w') - Q(w)|
    =
    |a_1| \times |w' - w|
    >
    0.
    \qedhere
  \end{equation}
\end{proof}

\begin{proof}[Proof of \cref{t:infinitely-attracted'}]
  Suppose ${-\infty < \chi < \log \ell(\cO)}$.
  This implies ${|\wmax(R)| > 0}$.
  On the other hand, ${|\wmax(R)| \le \theta}$.

  Put ${\kappa \= \log \ell(\cO) - \chi}$ and, for each~$z$ in~$\cO$,
  \begin{equation}
    \label{eq:117}
    r(z)
    \=
    \diam(D(z))
    \text{ and }
    u(z)
    \=
    \log \frac{|R'(z)| r(z)}{r(R(z))} - \log \ell(z) + \kappa.
  \end{equation}
  Note that
  \begin{equation}
    \label{eq:118}
    \kappa
    >
    0
    \text{ and }
    \sum_{z \in \cO} u(z)
    =
    0.
  \end{equation}
  Fix~$\whz_0$ in~$\cO$, let~$n_0$ in~$\{1, \ldots, \# \cO \}$ be maximizing ${\sum_{i = 0}^{n_0 - 1} u(R^i(\whz_0))}$, and, for every~$i$ in~$\Z_{> 0}$, put ${z_i \= R^{n_0 + i}(\whz_0)}$.
  Then, for every~$n$ in~$\Z_{\ge 0}$,
  \begin{equation}
    \label{eq:119}
    \sum_{i = 0}^n u(z_i)
    =
    \sum_{j = 0}^{n_0 + n} u(R^j(\whz_0)) - \sum_{j = 0}^{n_0 - 1} u(R^j(\whz_0))
    \le
    0.
  \end{equation}
  For ${n = 0}$, this implies
  \begin{equation}
    \label{eq:120}
    \frac{|R'(z_0)| r(z_0)}{r(z_1)}
    \le
    \ell(z_0) \exp(- \kappa)
    <
    \ell(z_0).
  \end{equation}
  If the immediate basin of~$\cO$ is of \textsc{Cantor} type, then ${\ell(z_0) \le |\wmax(R)|}$ and \cref{l:critical-criterion} implies that~$R$ has a critical point in~$X(z_0)$ and hence in~$\aA(z_0)$.
  If the immediate basin of~$\cO$ is not of \textsc{Cantor} type, then ${R(D(z_0)) = D(z_1)}$, ${\ell(z_0) \le |\deg_R(D(z_0))|}$, and \cref{l:critically-holomorphic} implies that~$R$ has a critical point in~$D(z_0)$ and hence in~$\aA(z_0)$.
  In all of the cases, $z_0$ belongs to~$\cO_{\crit}$.

  The next step is to prove that for every~$n$ in~$\Z_{> 0}$, there are~$j_n$ in ${\{0, \ldots, n - 1\}}$ and a critical point~$c_n$ of~$R$ in~$X(z_{j_n})$ such that~$R^{n - j_n - 1}$ is univalent on
  \begin{equation}
    \label{eq:121}
    \{z \in K \: |z - z_{j_n + 1}| \le |R(c_n) - z_{j_n + 1}|\},
  \end{equation}
  and
  \begin{equation}
    \label{eq:122}
    0
    <
    \frac{|R^{n - j_n}(c_n) - z_n|}{r(z_n)}
    \le
    \theta \exp \left( -n \kappa + \sum_{i = 0}^{n - 1} u(z_i) \right).
  \end{equation}
  To prove this for ${n = 1}$, put
  \begin{equation}
    \label{eq:123}
    \hvarrho_0
    \=
    \theta \exp(-\kappa + u(z_0)).
  \end{equation}
  Then, ${\hvarrho_0 < 1}$.
  If~$R(D(z_0))$ is strictly contained in~$D(z_1)$, then~$R$ has a pole at norm distance~$r(z_0)$ from~$z_0$ by \cref{l:disk-basin}, the immediate basin of~$\cO$ is of \textsc{Cantor} type, and
  \begin{equation}
    \label{eq:124}
    \ell(z_0)
    =
    |\wmax(R)|^{\deg_R(\aA(z_0))}
    \le
    |\wmax(R)|^{\deg_R(X(z_0))}.
  \end{equation}
  Combined with ${|\wmax(R)| \le \theta}$, this implies
  \begin{equation}
    \label{eq:125}
    |\wmax(R)|^{-(\deg_R(X(z_0)) - 1)} \frac{|R'(z_0)| r(z_0)}{r(z_1)}
    \le
    |\wmax(R)| \exp(-\kappa + u(z_0))
    \le
    \hvarrho_0
    <
    1.
  \end{equation}
  Hence, \cref{t:strictly-critically-mapped'} with ${Q(z) = R(z + z_0) - R(z_0)}$ and ${\whr = \hvarrho_0 r(z_1)}$ implies that~$R$ has a critical point~$c_0$ in~$X(z_0)$ satisfying
  \begin{equation}
    \label{eq:126}
    0
    <
    |R(c_0) - z_1|
    \le
    \hvarrho_0 r(z_1).
  \end{equation}
  This implies~\eqref{eq:122} with ${n = 1}$ and ${j_0 = 0}$ when~$R(D(z_0))$ is strictly contained in~$D(z_1)$.
  To complete the proof of~\eqref{eq:122} with ${n = 1}$, suppose ${R(D(z_0)) = D(z_1)}$.
  Note that ${X(z_0) = D(z_0)}$.
  By \textsc{Schwarz}' Lemma \cite[\S1.3.1]{Riv03c}, there is a real number~$\varrho_0$ satisfying ${\hvarrho_0 \le \varrho_0 < 1}$ and such that, putting ${D_0 \= \{z \in K \: |z - z_0| < \varrho_0 r(z_0)\}}$,
  \begin{equation}
    \label{eq:127}
    R(D_0)
    =
    \{z \in K \: |z - z_1| < \hvarrho_0 r(z_1)\}.
  \end{equation}
  If the immediate basin of~$\cO$ is of \textsc{Cantor} type, then, using ${|\wmax(R)| \le \theta}$,
  \begin{equation}
    \label{eq:128}
    \ell(z_0)
    =
    |\wmax(R)|^{\deg_R(\aA(z_0))}
    \le
    \theta |\deg_R(D_0)|^{\deg_R(\aA(z_0)) - 1}
    \le
    \theta |\deg_R(D_0)|^{\deg_R(D_0) - 1}.
  \end{equation}
  If the immediate basin of~$\cO$ is not of \textsc{Cantor} type and ${|\deg_R(D_0)| < 1}$, then
  \begin{equation}
    \label{eq:129}
    |\deg_R(D_0)|
    \le
    \theta
    \text{ and }
    \ell(z_0)
    \le
    |\deg_R(D_0)|^{\deg_R(D_0)}
    \le
    \theta |\deg_R(D_0)|^{\deg_R(D_0) - 1}.
  \end{equation}
  If the immediate basin of~$\cO$ is not of \textsc{Cantor} type and ${|\deg_R(D_0)| = 1}$, then
  \begin{equation}
    \label{eq:130}
    \ell(z_0)
    \le
    \theta
    =
    \theta |\deg_R(D_0)|^{\deg_R(D_0) - 1}.
  \end{equation}
  Thus, in all of the cases, ${\ell(z_0) \le \theta |\deg_R(D_0)|^{\deg_R(D_0) - 1}}$ and
  \begin{multline}
    \label{eq:131}
    |R'(z_0)|
    =
    \ell(z_0) \exp(-\kappa + u(z_0)) r(z_1)/r(z_0)
    \le
    |\deg_R(D_0)|^{\deg_R(D_0) - 1} \hvarrho_0 r(z_1)/r(z_0)
    \\ <
    |\deg_R(D_0)|^{\deg_R(D_0) - 1} \hvarrho_0 r(z_1)/(\varrho_0 r(z_0)).
  \end{multline}
  Combined with~\eqref{eq:15} and \cref{l:critically-holomorphic'}, this implies that~$R$ has a critical point~$c_0$ satisfying
  \begin{equation}
    \label{eq:132}
    |c_0 - z_0|
    <
    \varrho_0 r(z_0)
    \text{ and }
    0
    <
    |R(c_0) - z_1|
    <
    \hvarrho_0 r(z_1).
  \end{equation}
  This implies~\eqref{eq:122} with ${n = 1}$ and ${j_0 = 0}$ when ${R(D(z_0)) = D(z_1)}$, and completes the proof of~\eqref{eq:122} when ${n = 1}$.
  Let~$n$ in~$\Z_{> 0}$ be such that there are~$j_n$ in ${\{0, \ldots, n - 1\}}$ and a critical point~$c_n$ of~$R$ in~$X(z_{j_n})$ such that~$R^{n - j_n - 1}$ is univalent on~\eqref{eq:121} and~\eqref{eq:122} holds.
  Put
  \begin{equation}
    \label{eq:133}
    v_n
    \=
    R^{n - j_n}(c_n)
    \text{ and }
    B_n
    \=
    \{z \in K \: |z - z_n| \le |v_n - z_n|\}.
  \end{equation}
  By~\eqref{eq:119} and~\eqref{eq:122},
  \begin{equation}
    \label{eq:134}
    0
    <
    |v_n - z_n|
    \le
    \theta \exp(-n\kappa) r(z_n)
    <
    \theta r(z_n).
  \end{equation}
  If~$z_n$ is outside~$\cO_{\crit}$, then~$R$ is univalent on ${\{ z \in K \: |z - z_n| < \theta r(z_n)\}}$ by \cref{l:injectivity-radius}.
  Since the latter set contains~$B_n$,
  \begin{equation}
    \label{eq:135}
    |R^{n + 1 - j_n}(c_n) - z_{n + 1}|
    =
    |R(v_n) - z_{n + 1}|
    =
    |R'(z_n)| \times |v_n - z_n|
  \end{equation}
  and~$R^{n - j_n}$ is univalent on~\eqref{eq:121}.
  Together with~\eqref{eq:122}, this implies
  \begin{multline}
    \label{eq:136}
    0
    <
    \frac{|R^{n + 1 - j_n}(c_n) - z_{n + 1}|}{r(z_{n + 1})}
    =
    \frac{|R'(z_n)|r(z_n)}{r(z_{n + 1})} \times \frac{|v_n - z_n|}{r(z_n)}
    \le
    \exp(u(z_n) - \kappa) \frac{|v_n - z_n|}{r(z_n)}
    \\ \le
    \theta \exp \left( -(n + 1) \kappa + \sum_{i = 0}^n u(z_i) \right).
  \end{multline}
  This proves that~\eqref{eq:122} holds with~$n$ replaced by ${n + 1}$, ${j_{n + 1} = j_n}$, and ${c_{n + 1} = c_n}$ when~$z_n$ is outside~$\cO_{\crit}$.
  To complete the proof that~\eqref{eq:122} holds with~$n$ replaced by ${n + 1}$, suppose~$z_n$ belongs to~$\cO_{\crit}$.
  If~$R$ is univalent on~$B_n$, then, as in the previous case, $R^{n - j_n}$ is univalent on~\eqref{eq:121}, \eqref{eq:135} and~\eqref{eq:136} hold, and hence~\eqref{eq:122} holds with~$n$ replaced by ${n + 1}$, ${j_{n + 1} = j_n}$, and ${c_{n + 1} = c_n}$.
  Suppose~$R$ is not univalent on~$B_n$, put
  \begin{equation}
    \label{eq:137}
    Q(z)
    \=
    R(z + z_n) - z_{n + 1}
    \text{ and }
    \whr
    \=
    \ell(z_n)^{-1} |R'(z_n)| \times |v_n - z_n|,
  \end{equation}
  and denote by~$X$ the component of ${Q^{-1}(\{z \in K \: |z| < \whr\})}$ containing~$0$.
  Then, ${Q}$ has a zero different from~$0$ whose norm is bounded from above by ${|v_n - z_n|}$.
  On the other hand, by~\eqref{eq:119} and~\eqref{eq:122},
  \begin{multline}
    \label{eq:138}
    \whr
    =
    \exp(u(z_n) - \kappa) \frac{|v_n - z_n|}{r(z_n)} r(z_{n + 1})
    \le
    \theta \exp \left( -(n + 1) \kappa + \sum_{i = 0}^n u(z_i) \right) r(z_{n + 1})
    \\ <
    r(z_{n + 1}),
  \end{multline}
  and thus ${X \subset X(z_n)}$ and ${\deg_Q(X) \le \deg_R(\aA(z_n))}$.
  If the immediate basin of~$\cO$ is of \textsc{Cantor} type, then
  \begin{equation}
    \label{eq:139}
    \ell(z_n)
    =
    |\wmax(R)|^{\deg_R(A(z_n))}
    \le
    |\wmax(R)|^{\deg_Q(X)}.
  \end{equation}
  Combined with \cref{t:strictly-critically-mapped''} and~\eqref{eq:138}, this yields that~$R$ has a critical point~$c_{n + 1}$ in~$X(z_n)$ satisfying
  \begin{equation}
    \label{eq:140}
    0
    <
    \frac{|R(c_{n + 1}) - z_{n + 1}|}{r(z_{n + 1})}
    \le
    \frac{\whr}{r(z_{n + 1})}
    \le
    \theta \exp \left( -(n + 1) \kappa + \sum_{i = 0}^n u(z_i) \right).
  \end{equation}
  If the immediate basin of~$\cO$ is not of \textsc{Cantor} type, then ${X \subset X(z_n) = D(z_n)}$ and thus~$X$ is a ball.
  So,
  \begin{equation}
    \label{eq:141}
    \ell(z_n)
    \le
    |\deg_Q(X)|^{\deg_Q(X)}
  \end{equation}
  and, by~\eqref{eq:15}, \eqref{eq:22} in \cref{l:critically-holomorphic'} is satisfied with~$r$ replaced by~$\diam(X)$ and~$r'$ by~$\whr$.
  Combined with~\eqref{eq:138}, this yields that~$R$ has a critical point~$c_{n + 1}$ in~$X(z_n)$ satisfying~\eqref{eq:140}.
  This completes the proof that~\eqref{eq:122} holds with~$n$ replaced by ${n + 1}$ and ${j_{n + 1} = n}$.
  By induction, for every~$n$ in~$\Z_{> 0}$ there are~$j_n$ in ${\{0, \ldots, n - 1\}}$ and a critical point~$c_n$ of~$R$ in~$X(z_{j_n})$ such that~$R^{n - j_n - 1}$ is univalent on~\eqref{eq:121} and~\eqref{eq:122} holds.

  Since~$R$ has a finite number of critical points and
  \begin{equation}
    \label{eq:142}
    \theta \exp \left( -n \kappa + \sum_{i = 0}^{n - 1} u(z_i) \right)
    \to
    0
    \text{ as }
    n
    \to
    +\infty,
  \end{equation}
  there is~$n_0$ in~$\Z_{> 0}$ such that ${n_0 \ge j_{n_0} + \# \cO + 1}$.
  Put
  \begin{equation}
    \label{eq:143}
    z_{\ddag}
    \=
    z_{j_{n_0} + 1},
    c_{\ddag}
    \=
    c_{n_0},
    \text{ and }
    D_{\ddag}
    \=
    \{z \in K \: |z - z_{\ddag}| \le |R(c_{\ddag}) - z_{\ddag}|\}.
  \end{equation}
  Then, ${R(c_{\ddag}) \neq z_{\ddag}}$, $R^{\# \cO}$ is univalent on~$D_{\ddag}$, and thus
  \begin{equation}
    \label{eq:144}
    \diam(R^{\# \cO}(D_{\ddag}))
    =
    |\lambda| \diam(D_{\ddag})
    <
    \diam(D_{\ddag}).
  \end{equation}
  Combined with ${R^{\# \cO}(z_{\ddag}) = z_{\ddag}}$, this implies ${R^{\# \cO}(D_{\ddag}) \subset D_{\ddag}}$.
  It follows that the forward orbit of~$R(c_{\ddag})$ under~$R^{\# \cO}$ is infinite, and thus the orbit of~$c_{\ddag}$ under~$R$ is infinite.
  l This proves the desired assertion with ${c = c_{\ddag}}$ and ${D = D_{\ddag}}$.
\end{proof}

\section{Examples}
\label{s:examples}
Throughout this section, fix an integer~$d$ satisfying ${d \ge 2}$.

\subsection{Attracting cycles not attracting a critical point}
\label{ss:A-sharpness}
This section provides examples showing that \cref{t:critically-attracted} is sharp if ${\lambda(d) > 0}$.
When ${\lambda(d) = 0}$, this section provides examples of attracting cycles of arbitrary period and small multiplier attracting no critical point.

Suppose ${\lambda(d) > 0}$.
If ${\lambda(d) = 1}$, then \cref{t:critically-attracted} is sharp because its hypothesis holds for every attracting cycle.
Suppose ${\lambda(d) < 1}$, so the characteristic of~$K$ is zero and~$p$ is a prime number greater than or equal to~$d$.
Put
\begin{equation}
  \label{eq:145}
  P_0(z)
  \=
  \begin{cases}
    z^d
    & \text{if } p \mid d \text{ and } p^2 \not\mid d;
    \\
    z^p - p^d z^d
    & \text{if } p \not\mid d \text{ or } p^2 \mid d.
  \end{cases}
\end{equation}
The polynomial~$P_0$ has a nontrivial reduction, equal to the polynomial with coefficients in~$\tK$ given by~$\zeta^d$ or~$\zeta^p$.
In all of the cases, for every~$n$ in~$\Z_{> 0}$, the polynomial~$P_0$ has a cycle~$\cO$ of minimal period~$n$ contained in~$\{z \in K \: |z| = 1\}$.
Since ${|P_0| = |p|}$ on this set, the multiplier~$\lambda$ of~$\cO$ satisfies ${|\lambda| = |p|^n}$.
The cycle~$\cO$ does not attract the critical point~$0$, because this point is fixed.
When ${p \mid d}$ and ${p^2 \not\mid d}$, it is the only finite critical point of~$P_0$.
Suppose ${p \not\mid d}$ or ${p^2 \mid d}$, and let~$c$ be a critical point of~$P_0$ distinct from~$0$.
Then,
\begin{equation}
  \label{eq:146}
  dp^{d - 1} c^{d - p}
  =
  1,
  P_0(c)
  =
  c^p(1 - p/d),
  \text{ and }
  |P_0(c)|
  \ge
  |c|^p
  =
  |d p^{d - 1}|^{-\frac{p}{d - p}}.
\end{equation}
It follows that the orbit of~$c$ under~$P_0$ diverges to~$\infty$ and it is thus not attracted to~$\cO$.

Suppose ${\lambda(d) = 0}$, so the characteristic of~$K$ is positive, equal to~$p$, and~$p$ is a prime number greater than or equal to~$d$.
Given~$\eta$ in ${\MK \setminus \{ 0 \}}$, put
\begin{equation}
  \label{eq:147}
  P_1(z)
  \=
  \begin{cases}
    \eta z + z^d
    & \text{if } p \mid d;
    \\
    \eta z + z^p + \eta^{pd} z^d
    & \text{if } d \not\mid p.
  \end{cases}
\end{equation}
The polynomial~$P_1$ has a nontrivial reduction, equal to the polynomial~$\zeta^p$ with coefficients in~$\tK$.
It follows that, for every~$n$ in~$\Z_{> 0}$, the polynomial~$P_1$ has a cycle~$\cO$ of minimal period~$n$ contained in~$\OK$.
Since ${|P_1| = |\eta|}$ on this set, the multiplier~$\lambda$ of~$\cO$ satisfies ${|\lambda| = |\eta|^n}$.
In particular, ${\cO}$ is attracting.
When~$d$ is a multiple of~$p$, the polynomial~$P_1$ has no finite critical point, so~$\cO$ cannot attract one.
If~$d$ is not a multiple of~$p$, then an analysis of the \textsc{Newton} polygon reveals that the norm of every finite critical point of~$P_1$ is equal to~$|\eta|^{- \frac{pd - 1}{d - 1}}$.
This implies that every finite critical point escapes to infinity under the iteration of~$P_1$.
Thus, in all of the cases~$\cO$ attracts no critical point.

\subsection{Rational maps with large critical values}
\label{ss:C-sharpness}
Suppose ${\lambda(d) > 0}$, so the characteristic of~$K$ is zero or strictly larger than~$d$.

The following example shows that \cref{t:critically-mapped} is sharp if~$d$ is not a power of~$p$.
For such a~$d$, there is~$q$ in~$\{1, \ldots, d - 1\}$ satisfying ${|q| = \lambda(d)}$.
Let~$\varepsilon$ in~$K$ be such that ${0 < |\varepsilon| < |q/d|}$ and put
\begin{equation}
  \label{eq:148}
  Q_0(z)
  \=
  \frac{(1 + \varepsilon -\varepsilon z)^d}{z^q}.
\end{equation}
Then, ${Q_0(\infty) = \infty}$, $0$ is the only finite pole of~$Q_0$,
\begin{equation}
  \label{eq:149}
  Q_0(1)
  =
  1,
  Q_0'(1)
  =
  -(q + \varepsilon d),
  \text{ and }
  |Q_0'(1)|
  =
  |q|
  =
  \lambda(d).
\end{equation}
On the other hand, ${1 + \varepsilon^{-1}}$ is a critical point of~$Q_0$ and its image by~$Q_0$ is~$0$.
Moreover, the only finite critical point~$c_0$ of~$Q_0$ different from~$1 + \varepsilon^{-1}$ is given by ${c_0 = -\frac{q(1 + \varepsilon)}{\varepsilon (d - q)}}$.
Since ${|d - q| = |d|}$,
\begin{equation}
  \label{eq:150}
  |c_0| > 1,
  |1 + \varepsilon - \varepsilon c_0|
  =
  \left| \frac{d(1 + \varepsilon)}{d - q} \right|
  =
  1,
  \text{ and }
  |Q_0(c_0)|
  >
  1.
\end{equation}
Thus, equality holds in~\eqref{eq:8} with ${z_0 = 1}$ and ${v = 0}$.

The following examples show that \cref{c:critically-mapped} is sharp.
Suppose ${d \le p}$ and put
\begin{equation}
  \label{eq:151}
  Q_1(z)
  \=
  z - z^d.
\end{equation}
Then, ${\lambda(d) = |d|}$, ${|Q_1'(0)| = 1}$, and the norm of every zero of~$Q_1$ different from~$0$ is equal to~$1$.
On the other hand, every critical point~$c$ of~$Q_1$ satisfies
\begin{equation}
  \label{eq:152}
  c^{d - 1}
  =
  d^{-1}
  \text{ and }
  |Q_1(c)|
  =
  |d|^{-\frac{d}{d - 1}}
  =
  \gamma(d) |d|^{-1}.
\end{equation}
Thus, equality holds in~\eqref{eq:10} with ${z_0 = 0}$ and ${v = Q_1(c)}$.

Suppose ${d > p}$.
Let~$q$ be the largest integer in~$\{p, \ldots, d\}$ satisfying ${|q| = \lambda(d)}$, and choose~$\alpha$ in~$K$ satisfying ${|\alpha| = 1}$.
In the case where ${|q| = |p|}$, let~$m$ be the integer such that ${q = m p}$ and choose~$\alpha$ such that, in addition, every root~$\zeta_0$ of ${1 + (1 -\wtm) \zeta^{p - 1}}$ in~$\tK$ satisfies ${\talpha \zeta_0^{q - p} \neq \wtm}$.
Put
\begin{equation}
  \label{eq:153}
  Q_2(z)
  \=
  \begin{cases}
    z + \frac{1}{p} z^p - \frac{\alpha}{d} z^d
    & \text{if } d = q;
    \\
    z + \frac{1}{p} z^p - \frac{\alpha}{q} z^q + pz^d
    & \text{if } d > q.
  \end{cases}
\end{equation}
Note that ${|Q_2'(0)| = 1}$.
Since ${|\alpha| = 1}$, the map~$Q_2$ has ${p - 1}$ zeros of norm~$|p|^{\frac{1}{p - 1}}$, and its remaining zeros different from~$0$ have norm strictly greater than~$|p|^{\frac{1}{p - 1}}$.
On the other hand, ${Q_2}$ has ${q - 1}$ critical points of norm~$1$, and, by the choice of~$\alpha$, each of these critical points~$c$ satisfies
\begin{equation}
  \label{eq:154}
  |Q_2(c)|
  =
  |q|^{-1}
  =
  |\lambda(d)|^{-1}.
\end{equation}
When ${d = q}$, these are all of the critical points of~$Q_2$, and thus equality holds in~\eqref{eq:10} with ${z_0 = 0}$.
Suppose ${d > q}$.
Then, ${|d| > |q|}$ and every critical point~$c'$ of~$Q_2$ of norm different from~$1$ satisfies
\begin{equation}
  \label{eq:155}
  |c'|
  =
  |d p|^{- \frac{1}{d - q}}
  >
  1
  \text{ and }
  |Q_2(c')|
  =
  |q|^{-1} \times |c'|^q
  >
  |q|^{-1}
  =
  |\lambda(d)|^{-1}.
\end{equation}
Together with~\eqref{eq:154}, this proves that equality holds in~\eqref{eq:10} with ${z_0 = 0}$.

\bibliographystyle{alpha}

\end{document}